\documentclass[11pt,oneside]{amsart}
\usepackage[a4paper, total={6in, 8.5in}]{geometry}
\usepackage[mathscr]{euscript}
\usepackage{amsthm}
\usepackage{amssymb}
\usepackage{hyperref}
\usepackage{graphicx}
\usepackage{float}
\usepackage{bold-extra}

\newcommand{\m}[1]{{\uppercase {\mathbf {#1}}}}
\newcommand{\sm}[1]{{\mbox{\scriptsize {\uppercase {\bf {#1}}}}}}
\newcommand{\vr}[1]{{\uppercase {\mathcal{#1}}}}
\newcommand{\hsp}{\mathsf{HSP}}

\newcommand{\aut}[1]{{\mathrm{Aut}\:\m #1}}
\newcommand{\clo}[1]{{\mathrm{Clo}\:\m #1}}
\newcommand{\clon}[2]{{\mathrm{Clo}_{#1}\m #2}}
\newcommand{\pol}[1]{{\mathrm{Pol}\:\m #1}}
\newcommand{\poln}[2]{{\mathrm{Pol}_{#1}\m #2}}
\newcommand{\con}{{\rm Con\:}}
\newcommand{\cn}[1]{{\con\m {#1}}}
\newcommand{\Cn}[1]{{{\bf\con}\m {#1}}}
\newcommand{\Sub}[1]{{\mathrm{Sub}\:\m #1}}

\newcommand{\Sg}{{\mathrm{Sg}}}

\def\Dj{\mbox{\raise0.3ex\hbox{-}\kern-0.4em D}}

\newcommand{\lb}{\langle}
\newcommand{\rb}{\rangle}
\newcommand{\mt}{\wedge}

\newcommand{\usub}{\subseteq}

\newcommand{\subp}{{\sf {SP}}}
\newcommand{\sub}{{\sf {S}}}

\newtheorem{thm}{Theorem}
\newtheorem{lm}[thm]{Lemma}
\newtheorem{prp}[thm]{Proposition}
\newtheorem{cor}[thm]{Corollary}
\theoremstyle{definition}
\newtheorem{df}{Definition}
\newtheorem{bax}{Base Axiom}
\newtheorem{bax1}{Stronger Base Axiom}
\newtheorem{hax}{Homomorphism Axiom}
\newtheorem{rax}{Relational Axiom}
\newtheorem{rmk}{Remark}

\newtheorem{ex}{Example}

\newtheorem{claim}{\bf Claim}

\begin{document}

\title[Colored edges and binary absorption in minimal Taylor algebras]{The colored edge theory of A. Bulatov and binary absorption in minimal Taylor algebras}
\author[Z. Brady]{Zarathustra Brady}
\author[P. \Dj api\'c]{Petar \Dj api\'c}
\address{Department of Mathematics and Informatics\\ 
University of Novi Sad\\ Serbia}
\email{petarn@dmi.uns.ac.rs}
\author[P. Markovi\'c]{Petar Markovi\'c}
\address{Department of Mathematics and Informatics\\ 
University of Novi Sad\\ Serbia}
\email{pera@dmi.uns.ac.rs}
\author[A. Proki\'c]{Aleksandar Proki\'{c}}
\address{Faculty of Technical Sciences\\ 
University of Novi Sad\\ Serbia}
\email{aprokic@uns.ac.rs}
\author[V. Uljarevi\'c]{Vlado Uljarevi\'c}
\address{Department of Mathematics and Informatics\\ 
University of Novi Sad\\ Serbia}
\email{vlado.uljarevic@dmi.uns.ac.rs}
\thanks{P. \Dj api\'c, P. Markovi\'c and V. Uljarevi\'c were supported by the Ministry of Education, Science and Technological Development of the Republic of Serbia (Grants No. 451-03-137/2025-03/200125 and 451-03-136/2025-03/200125). Aleksandar Proki\'{c} was supported by the Ministry of Education, Science and Technological Development of the Republic of Serbia (Grant No. 451-03-137/2025-03/200156.)}


\begin{abstract}
We find a new definition of colored edge graphs of finite algebras in the case of minimal Taylor algebras, a definition which includes the graphs invented by A. Bulatov. Next we proceed to reprove the main results of A. Bulatov's theory in the case of minimal Taylor algebras and in our setting, finding several simplifications compared to the more general case of smooth algebras Bulatov considered.
\end{abstract}

\maketitle

\section{Introduction}

This paper, the first in a series of at least three, is concerned with the Constraint Satisfaction Problem, and its computational complexity. For three decades or so, the main focus in the investigation of computational complexity of the Constraint Satisfaction Problem was the Dichotomy Conjecture of Feder and Vardi, stated in \cite{FV}, which said that, for any finite relational structure $\mathbb{A}$ the complexity of $CSP(\mathbb{A})$ can be either tractable, or NP-complete. We start the investigation of the theories behind the two proofs of the Dichotomy Conjecture, by D. Zhuk \cite{Zhdich} and by A. Bulatov \cite{Budich}, and try to simplify them, connect them and find simpler methods than either by combining their ideas. The first effort towards unification of these proofs was in the seminal paper \cite{dreamteam}, where the minimal Taylor algebras were introduced as the proper setting in which Dichotomy can be analyzed. We will work within minimal Taylor algebras throughout our project.

In this paper our main focus is on simplifying the initial part of the colored edge theory of A. Bulatov. After distilling the important properties colored edges need to satisfy, we call them Edge Axioms, we will prove that Bulatov's edges satisfy them. We define another type of colored edges which also satisfy Edge Axioms, but are significantly easier to define than Bulatov's. Next, we prove that any colored graphs which satisfy Edge Axioms also satisfy the initial results of Bulatov's theory. We are confident that the remainder of Bulatov's theory, within the framework of minimal Taylor algebras, would be correct for any graphs satisfying the Edge Axioms (with Bulatov's original proofs). These initial results we prove are related to binary absorption, and we manage to simplify some of the proofs of Bulatov from \cite{BulatovGraph1}, \cite{BulatovGraph2} and a theorem from \cite{dreamteam}. In the penultimate section, we review the construction by M. Mar\'oti which uses unary polynomials and binary absorption, and we apply it to prove a connection between a partial result by Zhuk and another one by Bulatov. We finish the paper with an announcements of results we will prove in subsequent papers.

\section{Background and Notation}

\subsection{Universal algebra} We assume the reader has a thorough understanding of the basic notions of Universal Algebra, the readers who need help with basic Universal Algebra are advised to look up the missing parts in \cite{burris-sank}, or \cite{alvin1}, \cite{alvin2} and \cite{alvin3}. The notions of the centralizer condition and the centralizer (some call it the annihilator) from Commutator Theory, which are used in Section 8, are defined in the book \cite{comm}. In this paper we will deal with finite and idempotent algebras, so we assume that any algebra we mention is finite and idempotent, unless stated otherwise. Here we state the results and define the notions which are less standard, or more obscure, but will be used in the paper.

Following A. Bulatov, we define the hypergraph of proper subuniverses:

\begin{df}\label{H(A)def}
Let $\m a$ be an algebra. By $H(\m a)$ we denote the set of all proper subuniverses, i.e., $H(\m a):=\{B\subsetneq A:B\in \Sub{a}\}$.
\end{df}

We consider $H(\m a)$ as a hypergraph, so that we can use the related terminology of paths, connectedness etc. Next we recall the definition of a tolerance and A. Bulatov's notions of a link tolerance and link congruence.

\begin{df}\label{toldef}
Let $\m a$ be an algebra. A {\em tolerance} of $\m a$ is a reflexive, symmetric and compatible binary relation on $\m a$. The transitive closure of a tolerance is, of course, a congruence. When the transitive closure of the tolerance $\tau$ on $\m a$ is $A\times A$, we say that $\tau$ is {\em connected}. When the only tolerances on $\m a$ are the diagonal and the full relation, we say that $\m a$ is {\em tolerance-free}.
\end{df}

\begin{df}\label{linkdef}
Let $\m a_1,\dots,\m a_n$ be algebras and $\m r\leq_{sd}\m a_1 \times\dots \times \m a_n$. For any $i\leq n$, the $i$th {\em link tolerance} of $\m r$, denoted by $tol_iR$ is defined by
\[
\begin{gathered}
tol_iR:=\{(a,b)\in A_i\times A_i:(\exists a_1,\dots,a_{i-1},a_{i+1},\dots,a_n)\\
(a_1,\dots,a_{i-1},a,a_{i+1},\dots,a_n)\in R\text{ and }(a_1,\dots,a_{i-1},b,a_{i+1},\dots,a_n)\in R\}
\end{gathered}
\]
The transitive closure of $tol_iR$ is the $i$th link congruence of $\m r$, denoted by $lk_iR$.
\end{df}

Note that $tol_iR$ can be defined even if $R$ has no algebraic structure, but then it would be just a symmetric relation. Reflexivity of $tol_iR$ follows from subdirectness of $R$ and compatibility of $tol_iR$ from the compatibility of $R$.

We define the neighbors of a point just for binary relations, though the extension to $n$-ary relations is possible.

\begin{df}
Let $\m a$ and $\m b$ be idempotent algebras and $\m r\leq\m a \times \m b$. For any $a\in A$, the set of right $R$-neighbors of $a$, denoted by $[a]R$, are the set $\{x\in B:(a,x)\in R\}$. Similarly, the set of left $R$-neighbors of $b\in B$, denoted by $R[b]$ is $\{x\in A:(x,b)\in R\}$.
\end{df}

Both $[a]R$ and $R[b]$ are subuniverses of $\m b$ and $\m a$, respectively, which is proved by using primitive positive definitions (for short, pp-definitions) and idempotence.

\begin{prp}\label{conntolH}
Let $\m a$ be an idempotent algebra and $\tau$ a connected tolerance of $\m a$. Then either $\tau=A^2$, or $H(\m a)$ is a connected hypergraph.
\end{prp}

\begin{proof}
Assume $\tau\neq A^2$. We call the maximal subsets $B$ of $A$ which satisfy $B^2\subseteq\tau$ the tolerance classes. Each tolerance class is a proper subuniverse of $\m a$ since $B=\cap\{\tau[b]:b\in B\}$ and $\tau[b]$ is a subuniverse, as we mentioned before this Proposition. Moreover, for any pair $a,b\in A$ such that $(a,b)\in \tau$ there exists at least one tolerance class $B$ of $\tau$ such that $a,b\in B$. Hence, from connectedness of $\tau$ follows the connectedness of $H(\m a)$.
\end{proof}

\begin{thm}\label{connected-linked}
Let $\m a$ and $\m b$ be idempotent algebras and $\m r\leq_{sd}\m a\times\m b$ be such that $tol_1R$ is a connected tolerance. Then
\begin{enumerate}
\item $tol_2R$ is also a connected tolerance.
\item If there exists $b\in B$ such that $R[b]=A$, then $tol_1R=A^2$.
\item If $tol_1R\neq A^2$, then $H(\m a)$ is a connected hypergraph.
\item Assume that $A=\Sg(a_1,a_2)$ for some $a_1,a_2\in A$. In that case, $tol_1R=A^2$ iff there exists some $b\in B$ such that $R[b]=A$.
\end{enumerate}
\end{thm}

\begin{proof}
To prove (1), note that $tol_1R$ is a connected tolerance iff $R$, viewed as a bipartite graph between the disjoint unions of $A$ and $B$, is connected, iff $tol_2R$ is connected. (2) is obvious. (3) follows from Proposition~\ref{conntolH}. For the direction $\Rightarrow$ of (4), note that from $tol_1 R=A^2$ there must exist some $b\in B$ such that $(a_1,b),(a_2,b)\in R$. But then $A=\Sg(a_1,a_2)=R[b]$. The direction $(\Leftarrow)$ of (4) is statement (2).
\end{proof}

\subsection{CSP} The Constraint Satisfaction Problem (CSP for short) has several equivalent definitions, each offering slightly different language to express the same thing. We will use definitions from \cite{SMB2}, which were based on \cite{barto-collapsewidth}.

\begin{df}\label{constraintlanguage}
By a {\em constraint language} we mean a set $\Gamma$ of relations (of any arities) on the same nonvoid base set $A$.
\end{df}

\begin{df}\label{csp}
Given a constraint language $\Gamma$ on the set $A$, we define an instance of $CSP(\Gamma)$ as any ordered triple $(V,A,\vr c)$ where $V$ is called the set of variables, while $\vr c$ is a set whose elements are called constraints. Each constraint is an ordered pair $(S,R)$, where $S\subseteq V$ is the constraint scope, while $R\subseteq A^S$ is such that there exist an integer $n$ and a surjective mapping $\varphi:n\rightarrow S$ such that $R\circ \varphi\in\Gamma$. Here $R\circ \varphi =\{g\circ \varphi:g\in R\}$. $R$ is the constraint relation of the constraint $(S,R)$. A mapping $f:V\rightarrow A$ is a solution to the instance $(V,A,\vr c)$ of $CSP(\Gamma)$ if for every constraint $(R,S)\in \vr c$, $f|_S\in R$.
\end{df}

Note that the way Feder and Vardi postulated their conjecture implies that we may consider only the case when $\Gamma$ is finite, but we do not make that requirement here, yet. On the other hand, we do assume that the domain $A$ is finite throughout this paper. We will usually assume that the variable set $V$ is $[n]=\{1,2,\dots,n\}$. If we close $\Gamma$ under intersection, and also under permutation and identification of coordinates (if $\Gamma$ was finite, it will remain finite after those closures), then we may assume, without loss of generality, that different constraints have different scopes. This is because we may intersect all constraint relations with the same scope and replace all of those constraints with the single constraint. 

\begin{df}\label{klminimaldef}
Let $P$ be an instance of CSP$(\Gamma)$. We say that $P$ is $(k,l)$-minimal if
\begin{itemize}
\item for any $S\usub V$, $|S|\leq l$, there is precisely one $i$ such that $S_i=S$ ($l$-density) and
\item whenever $(S,R)$ and $(S',R')$ are constraints such that $S'\usub S$ and $|S'|\leq k$, then $R'=R|_{S'}$ ($k$-consistency).
\end{itemize}
\end{df}

By using the $(k,l)$-minimality algorithm (see e.g. \cite{barto-collapsewidth}), for any given fixed numbers $k\leq l$ we can transform, in polynomial time, a given instance of $CSP(\Gamma)$ into an equivalent $(k,l)$-minimal instance. Whenever an instance is at least $(1,1)$-minimal, the second condition implies that each constraint relation $R$ is a subdirect product of $\m a_j$, $j\in S$, where $(\{j\},A_j)$ are constraints in $\vr c$.

According to \cite{J}, the complexity of the $CSP(\Gamma)$ depends on the compatible operations (polymorphisms) of the relational structure $(A,\Gamma)$. Therefore, it suffices to verify the conjecture for $\Gamma=\subp_{fin}(\m a)$ for some finite algebra $\m a$. More precisely (to conform to our definitions) we may assume that $\Gamma=\bigcup\limits_{n=1}^\infty\sub(\m a^n)$ and we will denote such $\Gamma$ by $\Gamma(\m a)$. The impact of this assumption on $\Gamma$ is that now $\Gamma$ contains all full finite powers of $A$ and that $\Gamma$ is closed under intersections and products (along with permutations and identifications of variables assumed earlier), thus $\Gamma$ is a {\em relational clone}.

This justifies our construction where we move from an algebra $\m a$ to one of its term reducts $\m a'$. Namely, any relation compatible with all operations of $\m a$ is compatible with all term operations of $\m a'$, as well. Therefore, any instance of $CSP(\m a)$ is an instance of $CSP(\m a')$, so if we can solve any instance of $CSP(\m a')$ in polynomial time, we can do the same with instances of $CSP(\m a)$.

We may assume that the relational structure $(A,\Gamma)$ has no endomorphisms except for automorphisms. Namely, if $\varphi$ is an endomorphism which maps $A$ onto a proper subset, then the set of instances of $CSP(\Gamma)$ which have a solution is precisely the same as the set of instances of $CSP(\Gamma|_{\varphi(A)})$ on the domain $\varphi(A)$ which has a solution (in the nontrivial direction, just compose the solution of $CSP(\Gamma)$ with the endomorphism). Such relational structures for which all endomorphisms are automorphisms are called {\em cores}.

For $(A,\Gamma)$ a core, the complexity of $CSP(\Gamma)$ equals the complexity of $CSP(\Gamma^c)$, which is $\Gamma$ augmented with all one-element unary relations. If $\Gamma=\Gamma(\m a)$, then $\Gamma^c=\Gamma(\m a^{id})$, where $\m a^{id}$ is the idempotent reduct of $\m a$, i.e., $\m a^{id}$ is an algebra whose operations are all idempotent term operations in $\clo{a}$. So, we justified the focus on idempotent finite algebras and assume from now on that all algebras are such. For more details about the reductions in this and the previous paragraph, cf. Theorems 4.4 and 4.7 of \cite{BJK}.

As was proved in \cite{BJK}, if $\m a$ has no Taylor term (to be defined later) then $CSP(\Gamma(\m a))$ is NP-complete. In \cite{BJK} it is conjectured that in the converse case the $CSP(\Gamma(\m a))$ is in P. A proof of this, the Algebraic Dichotomy Conjecture, in the papers \cite{Zhdich} and \cite{Budich} thus also confirmed the original Dichotomy Conjecture by Feder and Vardi. It matters to us that, if $\m a$ is an algebra with a Taylor term, we can replace $CSP(\m a)$ with $CSP(\m a')$ such that $\m a'$ is a term reduct of $\m a$ which still contains some Taylor term. This is the idea behind minimal Taylor algebras, which we will explore in the next section.

\subsection{Multisorted CSP and templates} Already when we reduce an instance of $CSP(\m a)$ to a $(k,l)$-minimal one, we introduced subuniverses $A_i$ of $\m a$ for $1\leq i\leq n$. Instead of $R\leq A^S$, we can consider $R$ as a subdirect product of $\{\m a_i:i\in S\}$. For some reductions we will also need $A_i$ to be homomorphic images of $\m a$ (or of its subalgebra), or a retract of $\m a$ (or of its subalgebra).

Now we define a {\em template} of CSP.

\begin{df}\label{CSPtemplate}
A class $\vr t$ of isomorphism types of similar finite algebras is called a CSP template if $\vr t$ is closed under homomorphic images and subalgebras. If, additionally, a template is closed under unary polynomial retracts, we call it an M-template. 
\end{df}

Note that we do not allow isomorphic algebras to appear more than once in a template, which is why we speak of isomorphism types, rather than algebras themselves. The reason is, if the language of a finite algebra $\m a$ is finite, then there are only finitely many algebras, up to isomorphism, which can be obtained from $\m a$ by taking subalgebras, homomorphic images and retracts. This will allow us to reduce a multisorted CSP instance by making the template a smaller set of isomorphism types. Now we define the multisorted CSP instance.

\begin{df}\label{multisortedCSP}
Given a template $\vr t$, a (multisorted) instance $P$ of $CSP(\vr t)$ is the triple $(V, D, \vr c)$. Here $D=\{\m a_i:i\in V\}$, where each $\m a_i$ is isomorphic to some algebra in $\vr t$, is the tuple of domains, while $\vr c$ is the set of constraints. Each constraint is an ordered pair $(S,R)$, where $S\subseteq V$, while $R\leq \prod\limits_{i\in S}\m a_i$ is a subdirect product. A mapping $f\in\prod\limits_{i\in V}A_i$ is a solution to the instance $(V,D,\vr c)$ of $CSP(\vr t)$ if for every constraint $(S,R)\in \vr c$, $f|_S\in R$.
\end{df}

Of course, we can take as the template $\vr t(\m a)$, the set of isomorphism types of all finite algebras which can be obtained from a finite algebra $\m a$ by taking homomorphic images, subalgebras and retracts. Then each $(1,1)$-minimal instance of $CSP(\m a)$ is an instance of $CSP(\vr t(\m a))$. Moreover, if $\m a$ has a finite language, then $\vr t(\m a)$ is finite, as we said above.

\subsection{Taylor terms}

A term $t$ of an algebra $\m a$ is a Taylor term if it is idempotent and satisfies a finite set of equations which is not satisfied by any projection operation on a set with more than one element. W. Taylor made an equivalent, but more specific condition for Taylor terms, but we will work with cyclic terms, an even more restrictive condition, which was proved in \cite{bkcyclic} to be equivalent to Taylor terms within the finite algebras (and locally finite varieties of algebras). First we define a cyclic term:

\begin{df}\label{cyclicdef}
A term $t(x_1,\dots,x_n)$ is a cyclic term of an algebra $\m a$ iff the following two identities hold in $\m a$:
\[
\begin{gathered}
t(x,x,\dots,x)\approx x\text{ and}\\
t(x_1,x_2,\dots,x_n)\approx t(x_2,x_3,\dots,x_n,x_1).
\end{gathered}
\]
\end{df}

\begin{thm}[Theorem 4.2 of \cite{bkcyclic}]\label{Taylor=cyclic}
A finite algebra $\m a$ has a Taylor term iff $\m a$ has a cyclic term of any prime arity $p$ such that $p>|A|$.
\end{thm}

Subsequently, M. Siggers proved in \cite{siggers} that a finite algebra has a Taylor term iff it has a term with 6 variables which satisfies some specified (finite) set of identities, and K. Kearnes, P. Markovi\'c and R. McKenzie improved this to a four-variable term in \cite{KMM}. We will need just the existence of those two characterizations, so we omit their exact equations.

\begin{df}\label{affinedef}
Let $\m a$ be an algebra. We say that $\m a$ is a {\em set} if every term interprets in $\m a$ as some projection. We say that $\m a$ is {\em affine} if there exist a ring $\m r$ and a left $\m r$-module $\m m$ such that the clones of polynomial operations $\pol{a}$ and $\pol{m}$ are equal.
\end{df}

In the case of idempotent affine algebras, we may assume without loss of generality that the $\m r$-module $\m m$ is faithful, so for every term $t(x_1,\dots x_n)$ there exist elements $r_1,\dots,r_n\in R$ such that
\[
\begin{gathered}
t^{\m a}(x_1,\dots x_n)=r_1x_1+\dots+r_nx_n
\text{ and}\\
r_1+r_2+\dots+r_n=1.
\end{gathered}
\]

The following is a summary of several well-known results in the case of finite idempotent Taylor algebras:

\begin{thm}\label{affineeq}
Let $\m a$ be a finite idempotent Taylor algebra. The following are equivalent:
\begin{enumerate}
\item $\m a$ is affine.
\item There exists a ring $\m r$ and an $\m r$-module $\m m$ with universe $A$ such that the clone of all idempotent operations of $\m m$ and $\clo{a}$ are equal.
\item $\m a$ is quasi-affine.
\item $\m a$ is Abelian.
\item There exists a congruence $\Theta$ of $\m a^2$ such that the diagonal relation $0_A$ is a $\Theta$-class.
\end{enumerate}
\end{thm}

We give another condition which forces an algebra to be affine:

\begin{thm}\label{R3affine}
Let $\m a$ be a finite idempotent Taylor algebra and $R\leq \m a^3$. For every $a\in A$ define the relations $R_{a,1}:=\{(x,y)\in A^2:R(a,x,y)\}$, $R_{a,2}:=\{(x,y)\in A^2:R(x,a,y)\}$ and $R_{a,3}:=\{(x,y)\in A^2:R(x,y,a)\}$. If for every $a\in A$ all three relations $R_{a,1}$, $R_{a,2}$ and $R_{a,3}$ are graphs of bijections of the set $A$, then $\m a$ is affine.
\end{thm}

\begin{proof}
Another way to look at $R$ is to note that, for any $a,b\in A$ there exists a unique $c$ in $A$ such that $R(a,b,c)$, so $R$ is the graph of an operation from $A^2$ into $A$. Moreover, as there exist unique $d\in A$ such that $R(d,b,a)$ and unique $e\in A$ such that $R(b,e,a)$, thus $R$ is the graph of a quasigroup operation $r(x,y)$ on the set $A$, and such that $r$ commutes with all operations of $\m a$.

Define the relation $S(x_1,y_1,x_2,y_2)$ of $A^4$ by $(\exists z)R(x_1,y_1,z)\mt R(x_2,y_2,z)$. $S$ is compatible relation with $\m a$ (we have its pp-definition from $R$) and it is the kernel of the quasigroup operation $r:A^2\rightarrow A$. From all we know about $R$, it follows that $S$, viewed as a binary relation between its first two coordinates and its last two coordinates, is an equivalence relation on the set $A^2$, i.e., it is a congruence on $\m a^2$.

Fix some $a\in A$. Since $\m a$ is idempotent, $\{a\}$ is a subuniverse of $\m a$ and the bijection whose graph is $R_{a,3}(x,y)$ is pp-definable from compatible relations $\{a\}$ and $R$, so the bijection $R_{a,3}$ is an automorphism of $\m a$, denote it $\varphi$. So, if we define $S'(x_1,y_1,x_2,y_2):=S(\varphi(x_1),y_1,\varphi(x_2),y_2)$, it is another congruence of the algebra $\m a^2$. Moreover, one of its classes (the one where the quantified $z$ in the definition of $S$ is chosen to be $a$) is 
$$\{(\varphi(x_1),\varphi(x_1),\varphi(x_2),\varphi(x_2)):x_1,x_2\in A\}.$$ 
Since $\varphi\in \aut a$, it follows that the diagonal relation $0_A$ is a congruence class of the congruence on $\m a^2$ defined by $S'$. According to Theorem~\ref{affineeq} $(5)\Rightarrow(1)$, $\m a$ is affine. 
\end{proof}

A key concept to L. Barto's and M. Kozik's approach to Taylor algebras and the study of the Constraint Satisfaction problem is absorption, defined below.

\begin{df}[Definition 3.3 of \cite{dreamteam}]\label{absorbdef}
Let $\m a$ be an algebra and $B \subseteq A$. We call $B$ an $n$-\emph{absorbing set} of $A$ if there is a term operation $t \in \clon{n}{a}$ such that $t(\mathbf{a}) \in B$ whenever $\mathbf{a} \in A^n$ and $|\{i:\mathbf{a}(i) \in B\}| \geq n-1$.

If, additionally, $B$ is a subuniverse of $\m a$, we write $B\lhd_n \m a$, or $B\lhd \m a$ when the arity is not important.
\end{df}

A related notion to absorption is projectivity, defined below. Its variant was first defined as {\em cube term blockers} in \cite{MMM}.

\begin{df}[Definitions 2.2 and 5.6 of \cite{dreamteam}]
Let $\m a$ be an algebra and $B\subseteq A$. We say that $B$ is a {\em projective subuniverse} if for every $f\in\clon{n}{a}$ there exists a coordinate $i$ of $f$ such that $f(\mathbf{a})\in B$ whenever $\mathbf{a}\in A^n$ is such that $\mathbf{a}(i)\in B$. The set $B$ is a {\em strongly projective subuniverse} if for every $f\in\clon{n}{a}$ and every essential coordinate $i$ of $f$, we have $f(\mathbf{a})\in B$ whenever $\mathbf{a}\in A^n$ is such that $\mathbf{a}(i)\in B$. We say that $a\in A$ is an {\em absorbing element} if $\{a\}$ is a strongly projective subuniverse of $\m a$.
\end{df}

Some easy observations are that projective subuniverses must, indeed, be subuniverses, and that in every nontrivial clone any binary operation which is not a projection witnesses the fact that a strongly projective subuniverse is 2-absorbing.

\section{Minimal Taylor algebras}

By using either \cite{siggers}, \cite{KMM}, or even \cite{bkcyclic}, we can find, in each finite Taylor algebra $\m a$, a Taylor term which has either a fixed arity, or its arity (if we are applying \cite{bkcyclic}) depends only on the size $|A|$ of the universe. Therefore, any finite Taylor algebra must have as its term one of finitely many Taylor terms specified in the mentioned papers. Another way to restate the same fact is:\\
\smallskip

{\em On any finite set, there exist finitely many Taylor clones such that any Taylor clone contains one of them.}\\
\smallskip

Hence the definition of minimal Taylor algebras states:

\begin{df}\label{minTaylordef}
A finite algebra $\m a$ is a minimal Taylor algebra if $\m a$ has a Taylor term, and moreover, each Taylor term operation of the algebra $\m a$ generates the clone $\mathrm{Clo}{\m a}$ of all term operations of $\m a$.
\end{df}

By the remarks above Definition~\ref{minTaylordef}, we have

\begin{prp}[Proposition 5.2 of \cite{dreamteam}]
Every finite Taylor algebra has a term reduct which is a minimal Taylor algebra.
\end{prp}

The next two results of \cite{dreamteam} are given without proof. They are easy to prove, but fundamental results for anyone trying to apply minimal Taylor algebras.

\begin{thm}[Proposition 5.3 of \cite{dreamteam}]\label{subuniversesminTaylor}
Let $\m a$ be a minimal Taylor algebra, $t$ a term and $B\subseteq A$. If $B$ is closed under $t$ and if $B$, together with the restriction of $t^{\m a}$ to $B$ forms a Taylor algebra, then $B$ is a subuniverse of $\m a$.
\end{thm}

\begin{thm}[Proposition 5.4 of \cite{dreamteam}]\label{minTvariety}
Let $\m a$ be a minimal Taylor algebra and $\vr v=\mathsf{HSP}_{fin}(\m a)$ the pseudovariety generated by $\m a$. Then every $\m b\in\vr v$ is also a minimal Taylor algebra.
\end{thm}

Next we recall some definitions and results on absorption in minimal Taylor algebras. The first theorem is a little more difficult to prove than Theorems \ref{subuniversesminTaylor} and \ref{minTvariety}, but still not too bad, and we also cite it without proof.

\begin{thm}[Theorem 5.5 of \cite{dreamteam}]\label{abssetsubuniv}
Let $\m a$ be a minimal Taylor algebra and $B$ an absorbing set of $\m a$. Then $B$ is a subuniverse of $\m a$.
\end{thm}

The absorption is a much stronger property in minimal Taylor algebras than in general Taylor algebras, in which the above theorem has easy counterexamples. However, the binary absorption $\lhd_2$ is even stronger.

\begin{thm}[Theorem 5.7 of \cite{dreamteam}]\label{2absstrproj}
Let $\m a$ be a minimal Taylor algebra and $B\subseteq A$. The following are equivalent:
\begin{enumerate}
\item $B$ is a 2-absorbing subset of $\m a$.
\item $B$ is a projective subuniverse of $\m a$.
\item $B$ is a strongly projective subuniverse of $\m a$.
\end{enumerate}
\end{thm}

Since binary absorption is such a strong property in minimal Taylor algebras, it will be useful to have a criterion which establishes this property. Fortunately, there is such a result in \cite{dreamteam}, and we will provide an alternative proof of it, see Corollary~\ref{2abscriterion}, after we have proved a few more results and defined a few more concepts.

Another concept which we will investigate is the ternary absorption. There is a related concept, the {\em center} of an algebra, which is defined below.

\begin{df}\label{centerdef}
Let $\m a$ be a minimal Taylor algebra. $\emptyset\neq C\subseteq A$, is a \emph{center} (resp.~\emph{Taylor center}) if there exists an algebra (resp.~Taylor algebra) $\m b$ of the same signature as $\m a$ and $R\leq_{sd}\m a\times \m b$ such that $\m b$ has no nontrivial 2-absorbing subuniverse and that \[C=\{a\in A:[a]R=B\}.\]
\end{df}

In minimal Taylor algebras, we can say a lot more about the centers:

\begin{thm}[Theorem 5.10 of \cite{dreamteam}]\label{centers3abs}
Let $\m a$ be a minimal Taylor algebra and $B\subseteq A$. The following statements are equivalent:
\begin{enumerate}
\item $B\lhd_3 \m a$.
\item $(B\times A)\cup (A\times B)$ is a subuniverse of $\m a^2$.
\item $B$ is a Taylor center of $\m a$.
\end{enumerate}
\end{thm}

There is a result by D. Zhuk which connects binary absorbing subuniverses, ternary absorbing subuniverses, and congruence classes of ``polynomially complete" congruences with the Constraint Satisfaction Problem. This is the following theorem, which follows from Theorems 5.5 and 5.6 of \cite{Zhdich}.

\begin{thm}[Zhuk's Reduction Theorem]\label{Zhlemma}
Let $\vr t$ be a template of minimal Taylor algebras and let $P=(V,D,\vr c)$ be a multisorted instance of $CSP(\vr t)$. Suppose that $P$ is Z-irreducible, (1,1)-minimal, cycle consistent and has a solution.
\begin{enumerate}
\item If $B\lhd_2 \m a_i$ or  $B\lhd_3 \m a_i$, then $P$ has a solution $f$ such that $f(i)\in B$.
\item if $\m a_i$ has no proper binary nor ternary absorbing subuniverse, $\theta\in \Cn a_i$, $\m a_i/\theta$ is polynomially complete and $B$ is any fixed $\theta$-class, then $P$ has a solution $f$ such that $f(i)\in B$.
\end{enumerate}
\end{thm}

We do not define the Z-irreducible, (1,1)-minimal and cycle consistent instances, but we note that, in polynomial time, any instance can either be solved, or transformed to such an instance which is equivalent to the original one. We also don't define the polynomially complete congruences, but the following Zhuk's Four Types Theorem (for the case of minimal Taylor algebras) shows they are quite ubiquitous:

\begin{thm}[The Four Types Theorem]\label{4types}
Suppose that $\m a$ is a minimal Taylor algebra. Then at least one of the following four statements is correct:
\begin{enumerate}
\item There exists $B\subsetneq A$ such that $B\lhd_2\m a$.
\item There exists $B\subsetneq A$ such that $B\lhd_3\m a$.
\item There exists $\theta\in\Cn a$ such that $\theta\neq A^2$ and $\m a/\theta$ is polynomially complete.
\item There exists $\theta\in\Cn a$ such that $\theta\neq A^2$ and $\m a/\theta$ is affine.
\end{enumerate}
\end{thm}

Simplifying slightly, Zhuk's Reduction Theorem allows us to keep reducing the instance until we are left with an instance which has an affine factor in each $\m a_i$, and Zhuk comes up with a brilliant learning algorithm which emulates solving systems of linear equations over factor algebras by solving smaller CSP instances. However, Zhuk's proof of his Reduction Theorem (Theorem~\ref{Zhlemma}) is inseparable from the whole Dichotomy Theorem, so we are not shortening nor simplifying anything if we use Zhuk's Reduction Theorem.

Finally, we recall a result about minimal Taylor algebras which provides us with a special operation which has many interesting properties.

\begin{thm}[Theorem 5.23 from \cite{dreamteam}]\label{universalopMTA}
Let $\m a$ be a minimal Taylor algebra. There exists a ternary term $t(x,y,z)$ which satisfies the following:
\begin{enumerate}
\item For every $B\subseteq A$ such that $B\lhd_2\m a$, $B$ binary absorbs $\m a$ via each of the operations $t(x,x,y)$, $t(x,y,x)$ and $t(y,x,x)$.
\item For every $B\subseteq A$ such that $B\lhd_3\m a$, $B$ absorbs $\m a$ via $t(x,y,z)$.
\item For every subalgebra $\m b\leq \m a$ and $\theta\in\cn b$ such that $\m b/\theta$ is affine, $\m b/\theta\models t(x,y,z)\approx x-y+z$.
\end{enumerate}
\end{thm}

A careful reader of \cite{dreamteam} may complain that item (3) of Theorem~\ref{universalopMTA} does not match the third item of Theorem 5.23 of \cite{dreamteam}. Namely, Theorem 5.23 was speaking of ``affine edges", which means that the algebras $\m b/\theta$ are 2-generated, with additional minimality properties. Having said that, a review of the proof of Theorem 5.23 and Theorem 5.13(c) in \cite{dreamteam} reveals that the 2-generated property is never actually used, just a finiteness assumption, so Theorem~\ref{universalopMTA} was proved in \cite{dreamteam}, after all. We strengthen their result slightly to fit our needs, after an easy lemma.

\begin{lm}\label{homabsorb}
Let $\m a$ and $\m b$ be similar algebras, $\varphi:\m a\rightarrow\m b$ a surjective homomorphism.
\begin{enumerate}
\item If $D\lhd_k \m b$ and $C=\varphi^{-1}(D)$, then $C\lhd_k\m a$ witnessed by the same $k$-ary term that witnesses $D\lhd_k \m b$.
\item If $C\lhd_k \m a$ and $D=\varphi(C)$, then $D\lhd_k\m b$ witnessed by the same $k$-ary term that witnesses $C\lhd_k \m a$.
\end{enumerate}
\end{lm}

\begin{proof}
Let $t(x_1,\dots,x_k)$ be the term witnessing that $D\lhd_k \m b$. Let $a_1,\dots,a_k\in A$ satisfy that at most one $a_i\notin C$. Then $$\varphi(t(a_1,\dots,a_k))=t(\varphi(a_1),\dots,\varphi(a_k))\in D,$$
since for at most one $i\in [k]$, $\varphi(a_i)\notin D$. Therefore, $t(a_1,\dots,a_k)\in C$ and (1) is proved.

Let $t(x_1,\dots,x_k)$ be the term witnessing that $C\lhd_k \m a$. Select $b_1,\dots,b_k\in B$ such that at most one $b_i\notin D$. Then select $a_1,\dots,a_k\in A$ such that for all $i\in[k]$, $\varphi(a_i)=b_i$ and whenever $b_i\in D$, we select $a_i\in C$. Hence, $a_i\notin C$ for at most one $i\in [k]$, and thus $t(a_1,\dots,a_k)\in C$. Applying $\varphi$, we obtain $t(b_1,\dots,b_k)\in D$, and (2) is proved.
\end{proof}

\begin{cor}\label{universaloppseudovar}
Let $\m a$ be a minimal Taylor algebra and $\vr v$ the pseudovariety it generates. There exists a ternary term $t(x,y,z)$ which satisfies the following:
\begin{enumerate}
\item For every $\m b\in\vr v$, every $C\subseteq B$ such that $C\lhd_2\m b$, $C$ binary absorbs $\m b$ via each of the operations $t(x,x,y)$, $t(x,y,x)$ and $t(y,x,x)$.
\item For every $\m b\in\vr v$, every $C\subseteq B$ such that $C\lhd_3\m b$, $C$ absorbs $\m b$ via $t(x,y,z)$.
\item For every affine algebra $\m b\in\vr v$, $\m b\models t(x,y,z)\approx x-y+z$.
\end{enumerate}
\end{cor}

\begin{proof}
Since the variety $\vr w$ generated by $\m a$ is finitely generated, $\vr v$ is the class of finite algebras in $\vr w$ and contains all finitely generated free algebras in $\vr w$. In particular $\m f:=\m f_{\vr w}(x,y,z)$ is in $\vr v$, and  by Theorem~\ref{minTvariety}, this means that $\m f$ is a minimal Taylor algebra, as well. We select the term $t(x,y,z)$ which is provided by Theorem~\ref{universalopMTA} for the minimal Taylor algebra $\m f$. 

Now let $\m b\in \vr v$ and $C\lhd_2 \m b$. Pick any $b\in B$ and $c\in C$ and denote $B_1:=\Sg(b,c)$ and $C_1:=B_1\cap C$. $\m c_1$ is a homomorphic image of $\m f$ by some homomorphism $\varphi$ (for example, select $\varphi(x)=\varphi(y)=b$ and $\varphi(z)=c$). If $C_2:=\varphi^{-1}(C_1)$, then by Lemma~\ref{homabsorb} (1) we get $C_2\lhd_2\m f$. Using the properties of $t$ from Theorem~\ref{universalopMTA} (1) and Lemma~\ref{homabsorb} (2), we obtain that $C_1\lhd_2 \m b_1$ via each of the terms $t(x,x,y)$, $t(x,y,x)$ and $t(y,x,x)$. Hence $t(b,b,c)$, $t(b,c,b)$, $t(c,b,b)$, $t(c,c,b)$, $t(c,b,c)$ and $t(b,c,c)$ are all in $C_1\subseteq C$, as needed to complete the proof of (1).

The proof of (2) for $t$ is analogous to the proof of (1), with the following changes: We select $c_1,c_2\in C$ and $b\in B$, so $B_1:=\Sg(b,c_1,c_2)$ and $C_1:=C\cap B_1$. The surjective homomorphism $\varphi:\m f\rightarrow\m b_1$ can be determined by $\varphi(x)=b$, $\varphi(y)=c_1$ and $\varphi(z)=c_2$, for example. For $C_2:=\varphi^{-1}(C_1)$, we get $C_2\lhd_3\m f$ and this absorption can be witnessed by $t(x,y,z)$ so Lemma~\ref{homabsorb} (2) implies that $t$ witnesses $C_1\lhd_3\m b_1$. Hence, $t(c_1,c_2,b)$, $t(c_1,b,c_2)$ and $t(b,c_1,c_2)$ are all in $C_1\subseteq C$, as desired for (2).

To prove (3), let $a,b,c\in B$ and we need to prove that $t(a,b,c)=a-b+c$. Let $\m c:=\Sg(a,b,c)$, let $\varphi:\m f\rightarrow \m c$ be given by $\varphi(x)=a$,  $\varphi(y)=b$ and $\varphi(z)=c$. Denote by $\theta:=\ker\varphi$. Since $\m f/\theta\cong\m c$, $\m f/\theta$ is affine, and by Theorem~\ref{universalopMTA} (3), $t([x]_\theta,[y]_\theta,[z]_\theta)=[x]_\theta-[y]_\theta+[z]_\theta$. Applying the isomorphism $\varphi$, we obtain $t(a,b,c)=a-b+c$, as desired.
\end{proof}

\section{Edge Axioms}\label{sec:EA}

In this section we fix a pseudovariety $\vr v$ generated by a finite minimal Taylor algebra, and define the directed graphs $as(\m a)$ and $sm(\m a)$ for all algebras $\m a\in \vr v$.

We don't define our edges explicitly. Instead, we list several properties we call {\em Edge Axioms}, and any assignment of pairs of directed graphs each minimal Taylor algebra in $\vr v$ which satisfies these axioms are our edge-colored graphs.

Let $\vr v$ be a pseudovariety of minimal Taylor algebras generated by a single finite minimal Taylor algebra. For each $\m a\in \vr v$, fix $as(\m a)=(A;\rightarrow_{as})$ and $sm(\m a)=(A;\rightarrow_{sm})$, two directed graphs whose vertex set is the universe of $\m A$. We denote $\rightarrow_s:=\rightarrow_{as}\cap \rightarrow_{sm}$, $s(\m a):=(A;\rightarrow_s)$, and $asm(\m a):=(A;\rightarrow_{as}\cup \rightarrow_{sm})$. We say that the class of triples $\{(\m a,as(\m a),sm(\m a)):\m a\in \vr v\}$ is a (finitely generated) pseudovariety of minimal Taylor algebras with colored edges if the following {\em edge axioms} are satisfied:

\begin{bax}\label{ax1}
Let $\m a\in\vr v$, and let $\m b\leq \m a$ be an affine subalgebra. Then for all $a,b\in B$, $a\rightarrow_{as}b$.
\end{bax}

\begin{bax}\label{ax2}
Let $\m a\in \vr v$, and let $\m b\leq \m a$ be a subalgebra with a majority term. Then for all $a,b\in B$, $a\rightarrow_{sm}b$.
\end{bax}

\begin{bax}\label{ax3}
Let $\m a\in \vr v$, and for $a,b\in A$ let there exist a term $t(x,y)$ such that $t$ acts on the set $\{a,b\}$ as a semilattice operation with the absorbing element $b$. Then $a\rightarrow_s b$.
\end{bax}

\begin{hax}
Let $\m a,\m b\in \vr v$ and let $f:\m a\rightarrow \m b$ be a homomorphism. If $a\rightarrow_{as}b$ in $as(\m a)$ (respectively, $a\rightarrow_{sm}b$ in $sm(\m a)$), then $f(a)\rightarrow_{as}f(b)$ in $as(\m b)$ (respectively, $f(a)\rightarrow_{sm}f(b)$ in $sm(\m b)$). 
\end{hax}

In other words, Homomorphism Axiom 1 claims that each algebraic homomorphism between algebras in $\vr v$ is also a graph homomorphism of their as-graphs and their sm-graphs.

\begin{hax}
Let $\m a,\m b\in \vr v$  and let $f:\m a\rightarrow \m b$ be a homomorphism. If 
\[
\begin{gathered}
f(a)\rightarrow_{as}f(b)\text{ in }as(\m b),\, f(a)\rightarrow_{sm}f(b)\text{ in }sm(\m b),\\
\text{or }f(a)\rightarrow_s f(b)\text{ in }s(\m b),
\end{gathered}
\] 
then, for every $b'\in A$ such that $f(b')=f(b)$ and $\Sg(\{a,b'\})$ is minimal among $\{\Sg(\{a,c\}):c\in A\text{ and }f(c)=f(b)\}$, we have
\[a\rightarrow_{as}b'\text{ in }as(\m a),\, a\rightarrow_{sm}b'\text{ in }sm(\m a),\text{ or }a\rightarrow_s b'\text{ in }s(\m a),\] respectively.
\end{hax}

\begin{rax}
Let $\m a,\m b\in \vr v$ and let $R\leq \m a\times \m b$. If 
\[
a_1\rightarrow_{as} a_2\text{ in }\m a,\, b_1\rightarrow_{sm}b_2\text{ in } \m b,\, (a_1,b_2)\in R\text{ and }(a_2,b_1)\in R,
\]
then $(a_2,b_2)\in R$.
\end{rax}

\begin{rax}
Let $\m a,\m b\in \vr v$ and let $R\leq \m a\times \m b$. If 
\[
\begin{gathered}
a_1\rightarrow_{as} a_2\text{ in }\m a,\, b_1\rightarrow_{as}b_2\text{ in } \m b,\\ 
(a_1,b_1)\in R,\, (a_1,b_2)\in R\text{ and }(a_2,b_1)\in R,
\end{gathered}
\] 
then $(a_2,b_2)\in R$.
\end{rax}

\begin{rax}
Let $\m a,\m b,\m c\in \vr v$ and let $R\leq \m a\times \m b\times \m c$. If \[
\begin{gathered}
a_1\rightarrow_{sm} a_2\text{ in }\m a,\, b_1\rightarrow_{sm}b_2\text{ in } \m b,\, c_1\rightarrow_{sm}c_2\text{ in } \m c, \\
(a_1,b_2,c_2)\in R,\, (a_2,b_1,c_2)\in R\text{ and }(a_2,b_2,c_1)\in R,
\end{gathered}
\] 
then $(a_2,b_2,c_2)\in R$.
\end{rax}

The following easy consequence of the Edge Axioms justifies why, in our notation $\rightarrow_s$, $\rightarrow_{as}$ etc., we don't specify the algebra.

\begin{prp}\label{sameinsubalg}
Let $\{(\m a,as(\m a),sm(\m a)):\m a\in \vr v\}$ be a finitely generated pseudovariety of minimal Taylor algebras with colored edges, $\m a\in \vr v$, $\m b\leq \m a$ and $a,b\in B$. Then $a\rightarrow_{as}b$ in $as(\m a)$ iff $a\rightarrow_{as}b$ in $as(\m b)$ and $a\rightarrow_{sm}b$ in $sm(\m a)$ iff $a\rightarrow_{sm}b$ in $sm(\m b)$.
\end{prp}

\begin{proof}
Consider the identity embedding as a homomorphism $f:\m b\rightarrow \m a$. Then Homomorphism Axiom 1 immediately gives direction $(\Leftarrow)$, while Homomorphism Axiom 2 implies the direction $(\Rightarrow)$, considering that $f(b)=f(b')$ implies $b=b'$ (since $f$ is the identity map).
\end{proof}

\section{Graphs which satisfy the edge axioms exist}

In this section we will prove that, in the case of minimal Taylor algebras, A. Bulatov's edges, defined in \cite{BulatovGraph1}, satisfy our edge axioms. Next we will define pairs of directed graphs in a different way, and these graphs will also satisfy the edge axioms on a finitely generated pseudovariety of minimal Taylor algebras.

\subsection{Bulatov's edges satisfy Edge Axioms}

First of all, A. Bulatov's edges are defined for ``smooth" algebras, so first we have to verify that minimal Taylor algebras are ``smooth".

\begin{df}\label{smoothdef}
Let $\m a$ be a finite Taylor algebra. $\m a$ is {\em smooth} if, for any similar algebra $\m b$, any $a,b\in A$ and any homomorphism $\varphi:\Sg^{\m a}(a,b)\rightarrow \m b$, if either
\begin{enumerate}
\item[(a)] there exists a term $t(x,y)$ such that $(\{\varphi(a),\varphi(b)\};t^{\m b})$ is a 2-element semilattice, or 
\item[(b)] there exists a term $m(x,y,z)$ such that $(\{\varphi(a),\varphi(b)\};m^{\m b})$ is a 2-element majority algebra,
\end{enumerate}
then $\{x\in A:\varphi(x)\in\{\varphi(a),\varphi(b)\}\}$ is a subuniverse of $\m a$.
\end{df}

\begin{rmk}
Note that, in his definition of ``thick edges", A. Bulatov needed to consider the case when $\{\varphi(a),\varphi(b)\}$ had one operation which acted as a semilattice and another which acted as the majority operation. This case can't occur in minimal Taylor algebras, since if $\{\varphi(a),\varphi(b)\}$ has a semilattice operation, then it is a minimal Taylor algebra which has a 2-absorbing subuniverse, say $\{\varphi(b)\}$. Thus, by Theorem~\ref{2absstrproj}, $\{\varphi(b)\}$ is also a strongly projective subuniverse and the equations $m(\varphi(a),\varphi(a),\varphi(b))=m(\varphi(a),\varphi(b),\varphi(a))=m(\varphi(b),\varphi(a),\varphi(a))=\varphi(a)$ become impossible for any idempotent term $m(x,y,z)$.

Also note the following: If we denote by $C:=[a]_{\ker\varphi}$ and $D:=[b]_{\ker\varphi}$ in either of the cases (a) and (b) considered by Definition~\ref{smoothdef}, then $C,D, C\cup D=\Sg(a,b)$ are all subuniverses of $\m a$ and that $D\lhd_2(C\cup D)$ if case (a) applies and $t(\varphi(a),\varphi(b))=t(\varphi(b),\varphi(a))=\varphi(b)$, while $C\lhd_3(C\cup D)$ and $D\lhd_3(C\cup D)$ if case (b) applies.
\end{rmk}

\begin{prp}\label{minTsmooth}
Any finite minimal Taylor algebra is smooth.
\end{prp}

\begin{proof}
Let $\m c$ be the subalgebra of $\m b$ whose universe is $\varphi(\Sg^{\m a}(a,b))$. According to Theorem~\ref{minTvariety}, both $\Sg^{\m a}(a,b)$ and $\m c$ are minimal Taylor algebras. From the fact that the semilattice operation and the majority operation are both Taylor operations and Theorem~\ref{subuniversesminTaylor} follows that $\{\varphi(a),\varphi(b)\}$ is a subuniverse of $\m c$. Therefore, the inverse image $\varphi^{-1}(\{\varphi(a),\varphi(b)\})$ is a subuniverse of $\Sg^{\m a}(a,b)$, so $\m a$ is smooth.
\end{proof}

What follows are A. Bulatov's definitions of the three types of directed edges in smooth algebras. To distinguish them from the edges that satisfy our Edge Axioms, we will denote A. Bulatov's edges by $\rightarrow_s^B$, $\rightarrow_a^B$ and $\rightarrow_m^B$ for the semilattice, affine and majority edges, respectively.

\begin{df}\label{Bulsedgedef}
Let $\m a$ be a minimal Taylor algebra and $a,b\in A$ a pair of distinct elements. We say that $ab$ is a \emph{semilattice edge} (in the sense of Bulatov), and write $a\rightarrow_s^B b$, if there exists a term $t(x,y)$ such that $(\{a,b\};t)$ is a 2-element semilattice such that $t(a,b)=t(b,a)=b$.
\end{df}

We note that the semilattice edges $\rightarrow_s^B$ in A. Bulatov's smooth algebras depend on the choice of a binary operation $t$ (which Bulatov stipulates to have additional properties), in particular Bulatov claims that the orientation of a semilattice edge depends on this choice. In minimal Taylor algebras, though, a simple application of Proposition~\ref{minTsmooth} and Theorem~\ref{2absstrproj} gives

\begin{cor}\label{sedge2abs}
Let $\m a$ be a minimal Taylor algebra and $a,b\in A$ two distinct elements. Then the following are equivalent:
\begin{enumerate}
\item $a\rightarrow_s^B b$.
\item $\{a,b\}\in \Sub{a}$ and $\{b\}\lhd_2\{a,b\}$.
\item $\{a,b\}\in \Sub{a}$ and $\{b\}$ is a strongly projective subuniverse of $\{a,b\}$.
\end{enumerate}
\end{cor}

The truth of condition (3) above clearly depends only on the minimal Taylor algebra, not on any choice of term operations, since the strong projectivity is realized by any cyclic term of $\m a$ (all variables of a cyclic term are essential).

Also, Bulatov's edges $\rightarrow_s^B$ are really the only possible edges in pseudovarieties of minimal Taylor algebras, as seen in the next lemma.

\begin{lm}\label{therecanbonly1sedge}
Let $\vr v$ be any finitely generated pseudovariety of minimal Taylor algebras with pairs of digraphs such that the class $\{(\m a,\rightarrow_{as},\rightarrow_{sm}):\m a \in \vr v\}$ satisfies Edge Axioms. If $\m a \in \vr v$, $a,b\in A$ then $a\rightarrow_s b$ iff $a\rightarrow
_s^B b$.
\end{lm}

\begin{proof}
If $a\rightarrow_s b$, we apply Relational Axiom 1 to $B=\Sg(a,b)$ and $R:=\Sg^{\sm a}((a,b),(b,a))$ (in order to have $\m r\leq_{sd}\m b\times\m b$), and we get $(b,b)\in R$. Hence, there exists a binary term $t(x,y)$ which acts on $\{a,b\}$ as the semilattice with absorbing element $b$, i.e., $a\rightarrow_s^B b$. 

On the other hand, from $a\rightarrow_s^B b$ and Theorem~\ref{subuniversesminTaylor}, follows that $\{a,b\}$ is a subuniverse of $\m a$ which is a two-element semilattice with absorbing element $b$, so Base Axiom 3 implies $a\rightarrow_s b$.
\end{proof}

Before we define the majority edges $\rightarrow_m^B$ and affine edges $\rightarrow_a^B$, we need the following definitions.

\begin{df}\label{majconddef}
A term operation $t(x,y,z)$ is said to satisfy the {\em majority condition} on a minimal Taylor algebra $\m a$ if
\begin{enumerate}
\item[(a)] For any $a,b\in A$ and any surjective homomorphism $\varphi$ which maps $\Sg^{\m a}(a,b)$ onto a two-element majority algebra, $t$ interprets as the majority operation on $\{\varphi(a),\varphi(b)\}$, and
\item[(b)] $\m a\models t(x,t(x,y,y),t(x,y,y))\approx t(x,y,y)$.
\end{enumerate}
\end{df}

\begin{df}\label{minconddef}
A term operation $t(x,y,z)$ is said to satisfy the {\em minority condition} on a minimal Taylor algebra $\m a$ if
\begin{enumerate}
\item[(a)] For any $a,b\in A$ and any surjective homomorphism $\varphi$ which maps $\Sg^{\m a}(a,b)$ onto an affine algebra $\m b$, $t(x,y,z)$ interprets as the operation $x-y+z$ on $\m b$, and
\item[(b)] $\m a\models t(t(x,y,y),y,y)\approx t(x,y,y)$.
\end{enumerate}
\end{df}

Note that, because of idempotence, the surjectivity of the homomorphism $\varphi$ implies that $\varphi(a)\neq\varphi(b)$. In the case of the majority condition, this implies that $\{\varphi(a),\varphi(b)\}$ is the whole $\varphi$-image of $\Sg(a,b)$.

We recall following result of A. Bulatov from \cite{BulatovGraph1}:

\begin{thm}[Corollary 22 and Lemma 23 of \cite{BulatovGraph1}]\label{Bunified}
Let $\vr k$ be a finite class of finite smooth algebras. Then there
exist ternary term operations $g$ and $h$ of $\vr k$ such that $g$ satisfies the majority condition and $h$ satisfies the minority condition on all algebras in $\vr k$. Moreover, we can find such $g$ and $h$ which, besides satisfying the majority and the minority condition, respectively, for every $\m a\in \vr k$ and $a,b\in A$ they satisfy the following:
\begin{enumerate}
\item[(a)] If $\varphi$ maps $\Sg(a,b)$ onto a two-element meet semilattice, then both $g(x,y,z)$ and $h(x,y,z)$ act on $\varphi(\Sg(a,b))$ as $x\mt y\mt z$.
\item[(b)] If $\varphi$ maps $\Sg(a,b)$ onto the two-element majority algebra, then $h$ acts on $\varphi(\Sg(a,b))$ as the first projection.
\item[(c)] If $\varphi$ maps $\Sg(a,b)$ onto an affine algebra, then $g$ acts on $\varphi(\Sg(a,b))$ as the first projection.
\end{enumerate}
\end{thm}

We plan to use A. Bulatov's results about his edges mostly as blackbox here, but in order to do that we need to prove they apply to a pseudovariety of minimal Taylor algebras. Namely, Bulatov extended his edges from a single smooth algebra to a finite class of smooth algebras, as we saw in Theorem~\ref{Bunified}, while the pseudovariety in which we have defined our Edge Axioms is an infinite set of isomorphism types. So we need the following result.

\begin{thm}\label{majmincondV}
Let $\m a$ be a finite minimal Taylor algebra and $\vr v$ the pseudovariety of finite minimal Taylor algebras $\m a$ generates. If $|A|^2=n$, and $t(x,y,z)$ satisfies the majority (respectively, minority) condition on any algebra in $\mathsf{HS}(\m a^n)$, then $t(x,y,z)$ satisfies the majority (resp., minority) condition on any algebra $\m b\in\vr v$.

If $t(x,y,z)$ satisfies the additional conditions (a)-(c) of Theorem~\ref{Bunified} on $\mathsf{HS}(\m a^n)$, then it satisfies the same conditions throughout $\vr v$.
\end{thm}

\begin{proof}
Let $\m b\in\vr v$ and $a,b\in B$ be such that $\varphi$ is a surjective homomorphism which maps $\Sg^{\m b}(a,b)$ onto the two-element majority algebra. Since $\m b\in \hsp_{fin}(\m a)$, there exists some positive integer $k$, some $\m c\leq \m a^k$ and some surjective homomorphism $\psi:\m c\rightarrow\m b$. Select two elements $a',b'\in C$ such that $\psi(a')=a$ and $\psi(b')=b$. Clearly, $\Sg^{\m c}(a',b')=\Sg^{\m a^k}(a',b')$. Therefore, $\varphi\circ\psi$ maps $\Sg^{\m a^k}(a',b')$ onto the two-element majority algebra.

Now we inductively select coordinates $i_1,i_2,\dots$ in $k$ in the following way: Let $a'(i_1)\neq b'(i_1)$ and if $i_1,i_2,\dots,i_j$ are selected, and, if such exist, select some $c,d\in\Sg^{\m a^k}(a',b')$. $c\neq d$, such that $\mathrm{pr}_{i_1,i_2,\dots,i_j}c=\mathrm{pr}_{i_1,i_2,\dots,i_j}d$. Then select $i_{j+1}$ so that $c(i_{j+1})\neq d(i_{j+1})$. Hence, $\mathrm{pr}_{i_1,i_2,\dots,i_j}\Sg(a',b')$ will keep increasing as $j$ increases, until the projection becomes the same size as $|\Sg(a',b')|$. This will happen at most when $j=n$, since there are at most $n$ many possible values for $(a'(i),b'(i))$, and when $a'(i)=a'(i')$ and $b'(i)=b'(i')$, then $c(i)=c(i')$ and $d(i)=d(i')$, so the coordinate $i'$ is redundant. So, for some selection $I=\{i_1,i_2,\dots,i_j\}$ of coordinates, $j\leq n$, we obtain an isomorphism $\tau$ which maps the projection $\mathrm{pr}_I\Sg(a',b')$ onto $\Sg(a',b')$. Denote $a''=\mathrm{pr}_I a'$ and $b''=\mathrm{pr}_I b'$. We get that $\varphi\circ\psi\circ\tau$ maps $\Sg^{\m a^I}(a'',b'')$ onto the two-element majority algebra with $\varphi\circ\psi\circ\tau(a'')=\varphi(a)$ and $\varphi\circ\psi\circ\tau(b'')=\varphi(b)$. Since $\m a^I\in\mathsf{HS}(\m a^n)$, thus the two-element majority algebra on $\{\varphi(a),\varphi(b)\}$ is also in $\mathsf{HS}(\m a^n)$, and since $t(x,y,z)$ satisfies the majority condition for all algebras in $\mathsf{HS}(\m a^n)$, we get that $t$ acts as the majority operation on $\{\varphi(a),\varphi(b)\}$, thus fulfilling condition (a) of Definition~\ref{majconddef}. Condition (b) of Definition~\ref{majconddef} is an identity and since it holds on $\m a\in\mathsf{HS}(\m a^n)$, it must hold throughout $\vr v$.

The statement about the minority condition has the same proof as above, we just replace the phrase ``the two-element majority algebra" with ``an affine algebra" everywhere. If $t(x,y,z)$ satisfies the additional conditions (a)-(c) of Theorem~\ref{Bunified} on $\mathsf{HS}(\m a^n)$, the proof that it satisfies the same conditions throughout $\vr v$ is analogous as we are once again checking homomorphic images of 2-generated subuniverses.
\end{proof}

\begin{cor}\label{majminexist}
Let $\m a$ be a finite minimal Taylor algebra and $\vr v$ the pseudovariety of finite minimal Taylor algebras $\m a$ generates. Then there exist ternary term operations $g$ and $h$ such that $g$ satisfies the majority condition and $h$ satisfies the minority condition on all algebras in $\vr k$ and they also satisfy the additional conditions given in Theorem~\ref{Bunified} (a)-(c).
\end{cor}

\begin{proof}
Let $n=|\m a|^2$. Note that $\mathsf{HS}(\m a^n)$ consists of finitely many isomorphism types of algebras, which are all smooth according to Theorem~\ref{minTvariety} and Proposition~\ref{minTsmooth}. According to Theorem~\ref{Bunified}, we can select $g$ and $h$ such that $g$ satisfies the majority condition and $h$ satisfies the minority condition on all algebras in $\mathsf{HS}(\m a^n)$, as well as the additional conditions of Theorem~\ref{Bunified} (a)-(c). But then Theorem~\ref{majmincondV} guarantees that $g$ satisfies the majority condition and $h$ satisfies the minority condition on all algebras in $\vr v$, and they also satisfy the additional conditions (a)-(c) on all algebras in $\vr v$.
\end{proof}

Now we are ready to define the majority and affine edges in the sense of Bulatov. In view of the results proved so far in this section, we feel justified to fix a finitely generated pseudovariety of minimal Taylor algebras $\vr v$, as well as two ternary terms $g(x,y,z)$ and $h(x,y,z)$ which satisfy the conclusion of Corollary~\ref{majminexist}.

\begin{df}\label{Buledgemadef}
Let $\m a\in \vr v$ be a finite minimal Taylor algebra and $a,b\in A$. 
\begin{enumerate}
\item[(a)] We say that $ab$ is a majority edge in the sense of Bulatov, and write $a\rightarrow_m^B b$, if for every term $t(x,y,z)$ which satisfies the majority condition on all algebras in $\vr v$, $b$ is an element of $\Sg(a,t(a,b,b))\cap \Sg(a,t(b,a,b))\cap\Sg(a,t(b,b,a))$. $a\rightarrow_m^B b$ is {\em special} if, additionally, there exists a surjective homomorphism from $\Sg(a,b)$ onto the two-element majority algebra.
\item[(b)] Let $h(x,y,z)$ be the fixed term which satisfies the minority condition, as well as conditions (a) and (b) of Theorem~\ref{Bunified}, on $\vr v$. We say that $ab$ is an affine edge in the sense of Bulatov, and write $a\rightarrow_a^B b$, if $h(b,a,a)=b$ and for every term $t(x,y,z)$ which satisfies the minority condition on all algebras in $\vr v$, $b\in \Sg(a,t(a,a,b))$.
\end{enumerate}
\end{df}

Now we verify that Bulatov's edges satisfy the Edge Axioms. In the pseudovariety $\vr v$, for all $\m a\in \vr v$ and $a,b\in A$, we define $a\rightarrow_{as}b$ iff $a\rightarrow_a^B b$ or $a\rightarrow_s^B b$, while $a\rightarrow_{sm}b$ iff $a\rightarrow_s^B b$ or $a\rightarrow_m^B b$. So we defined the digraphs $as(\m a)=(A;\rightarrow_{as})$ and $sm(\m a)=(A;\rightarrow_{sm})$ for all $\m a\in \vr v$. We will prove that the class of triples $\{(\m a,as(\m a),sm(\m a)):\m a\in\vr v\}$ satisfies the Edge Axioms.

Before we start, we have to make a sanity check regarding the s-edges. In order to prove it, we first restate a result from \cite{BulatovGraph2}.

\begin{lm}[Lemma 9 (2) of \cite{BulatovGraph2}]\label{diffedges}
Let $\m a_1,\m a_2\in \vr v$, $a\rightarrow_{asm}^B b$ in $\m a_1$ and $c\rightarrow_{asm}^B d$ in $\m a_2$. If the types of the two edges are different, then there exists a term $p(x,y)$ such that $p(a,b)=b$ and $p(d,c)=d$.
\end{lm}

Now comes the promised sanity check. The result stronger than the one stated below is actually proved in the proof of Lemma 24 (1) of \cite{BulatovGraph2}, but, since the statement of that lemma is weaker, we reproduce the proof for reader's convenience.

\begin{lm}\label{as-smimpliess}
For any $\m a\in\vr v$ and $a,b\in A$, $a\rightarrow_s^B b$ iff $a\rightarrow_{as}^B b$ and $a\rightarrow_{sm}^B b$.
\end{lm}

\begin{proof}
The only nontrivial case to check is $(\Leftarrow)$ when $a\rightarrow_a^B b$ and also $a\rightarrow_m^B b$. Let $p(x,y)$ be the binary term provided by Lemma~\ref{diffedges} for edges $a\rightarrow_a^B b$ and $a\rightarrow_m^B b$. Then $p(a,b)=b$ and $p(b,a)=b$, so $a\rightarrow_s^B b$.
\end{proof}

Now we verify that Bulatov's edges satisfy the Edge Axioms.

\begin{thm}
If $\vr v$ is a finitely generated pseudovariety of minimal Taylor algebras, and for all $\m a\in\vr v$, the digraphs $as(A)$ and $sm(\m a)$ are defined as before Lemma~\ref{diffedges}, then $\{(\m a,as(\m a),sm(\m a)):\m a\in\vr v\}$ satisfies the Edge Axioms.
\end{thm}

\begin{proof}
{\bf Base Axiom 1.} Assume that $\m a\in\vr v$, that $\m b\leq \m a$ is an affine subalgebra and $a,b\in B$. Then $C:=\Sg^{\m a}(a,b)$ is a subuniverse of $\m b$, so it is also an affine algebra. By taking as $\varphi$ the identity map of $C$, we conclude that any operation $h(x,y,z)$ which satisfies the minority condition on $\vr v$ must act on $C$ as $x-y+z$. So for any $h$ which satisfies the minority condition on $\vr v$ (including when we take for $t$ the fixed term $h(x,y,z)$), we have $b=t(b,a,a)=t(a,a,b)$ and therefore $a\rightarrow_a^B b$, implying $a\rightarrow_{as}b$.

{\bf Base Axiom 2.} Assume that $\m a\in\vr v$, that $\m b\leq \m a$ is a majority subalgebra and $a,b\in B$. Then $\Sg^{\m a}(a,b)=\{a,b\}$ and it is a two-element majority algebra. Using as $\varphi$ the identity map on $\{a,b\}$, we get that any operation $t(x,y,z)$ which satisfies the majority condition on $\vr v$ must act on $\{a,b\}$ as majority. So for any $t$ which satisfies the majority condition on $\vr v$, $b=t(a,b,b)=t(b,a,b)=t(b,b,a)$ and therefore $a\rightarrow_m^B b$, implying $a\rightarrow_{sm}b$.

{\bf Base Axiom 3.} Assume that $\m a\in\vr v$, $a,b\in A$ and that the binary term $t$ satisfies $t(a,a)=a$ and $t(a,b)=t(b,a)=t(b,b)=b$. Then $a\rightarrow_s^B b$ by Definition~\ref{Bulsedgedef}, and therefore $a\rightarrow_s b$ by Lemma~\ref{as-smimpliess}.

To prove {\bf Homomorphism Axioms 1 and 2}, we assume that $\m a,\m b\in\vr v$, that $\varphi:\m a\rightarrow\m b$ is a homomorphism and that $a,b\in A$. If $a\rightarrow_a^B b$, $a\rightarrow_m^B b$, or $a\rightarrow_s^B b$, then we need to prove $\varphi(a)\rightarrow_a^B \varphi(b)$, $\varphi(a)\rightarrow_m^B \varphi(b)$, or $\varphi(a)\rightarrow_s^B \varphi(b)$, respectively. But this is precisely the content of Lemma 11 (2) of \cite{BulatovGraph2}. Homomorphism Axiom 1 follows directly. On the other hand, if $\varphi(a)\rightarrow_a^B \varphi(b)$, $\varphi(a)\rightarrow_m^B \varphi(b)$, or $\varphi(a)\rightarrow_s^B \varphi(b)$, then we need to prove that there exists $b'\in A$ such that $\varphi(b')=\varphi(b)$ and moreover, that $a\rightarrow_a^B b'$, $a\rightarrow_m^B b'$, or $a\rightarrow_s^B b'$, respectively. But this is precisely the content of Lemma 11 (1) of \cite{BulatovGraph2} and Homomorphism Axiom 2 follows.

{\bf Relational Axiom 1.} Assume that $\m a,\m b\in\vr v$, that $\m r\leq_{sd}\m a\times\m b$ is a subdirect product, that $a_1\rightarrow_{as} a_2)$, that $b_1\rightarrow_{sm} b_2$, and that $(a_1,b_2),(a_2,b_1)\in R$. By the way we defined $as(\m a)$ and $sm(\m b)$, this means that either $a_1\rightarrow_s^B a_2$ and $b_1\rightarrow_s^B b_2$, or that $a_1\rightarrow a_2$ and $b_1\rightarrow b_2$ are Bulatov's thin edges of different types. In the first case, if $c(x_1,\dots,x_n)$ is any cyclic term of $\vr v$, then $c^{\m a\times\m b}((a_1,b_2),(a_1,b_2),\dots,(a_1,b_2),(a_2,b_1))=(a_2,b_2)$, using Corollary~\ref{sedge2abs} and the fact that every position in a cyclic term is essential. In the second case, by Lemma~\ref{diffedges} we get a binary term $p(x,y)$ such that $p((a_1,b_2),(a_2,b_1))=(a_2,b_2)\in R$. In both cases, Relational Axiom 1 holds.

{\bf Relational Axiom 2.} Assume that $\m a,\m b\in\vr v$, that $\m r\leq_{sd}\m a\times\m b$ is a subdirect product, that $a_1\rightarrow_{as} a_2$ in $\m a$, that $b_1\rightarrow_{as} b_2$ in $\m b$, and that $(a_1,b_1),(a_1,b_2),(a_2,b_1)\in R$. By the way we defined $\rightarrow_{as}$ in $\vr v$, one possibility is that $a_1\rightarrow_a^B a_2$ in $\m a$ and $b_1\rightarrow_a^B b_2$ in $\m b$, in which case $(a_2,b_2)\in R$ by Lemma 24 (1) of \cite{BulatovGraph2}. The remaining possibility is that $a_1\rightarrow_{as} a_2$ in $\m a$ and $b_1\rightarrow_{as} b_2$ in $\m b$ are Bulatov's thin edges which are in the cases already covered by Relational Axiom 1 (either different thin edges, or both are $\rightarrow_s^B$). By the proof of Relational Axiom 1, we obtain $(a_2,b_2)\in R$ even when $(a_1,b_1)\in R$ is not assumed. In both cases, Relational Axiom 2 holds.

{\bf Relational Axiom 3.} We assume that $\m a,\m b,\m c\in \vr v$ and that $R\leq \m a\times \m b\times \m c$ is a subdirect product. Also, let
$a_1\rightarrow_{sm}^B a_2$ in $\m a$, $b_1\rightarrow_{sm}^B b_2$ in $\m b$, $c_1\rightarrow_{sm}^B c_2$ in $\m c$, and $(a_1,b_2,c_2),(a_2,b_1,c_2),(a_2,b_2,c_1)\in R$. Again we have two cases. One case is when there exists an edge $u_1\rightarrow_s^B u_2$, for $u\in \{a,b,c\}$. WOLOG, assume that $a_1\rightarrow_s^B a_2$. Then consider $S=\{(x,y)\in A\times B:(x,y,c_2)\in R\}$. $S$ is a subuniverse of $\m a\times\m b$, being pp-definable with constants from $R$, and $(a_1,b_2),(a_2,b_1)\in S$. Moreover, $(a_1,a_2)\in as(\m a)$ and $(b_1,b_2)\in sm(\m b)$, so by Relational Axiom 1, $(a_2,b_2)\in S$. By definition of $S$, $(a_2,b_2,c_2)\in R$. The remaining case is when $a_1\rightarrow_m^B a_2$, $b_1\rightarrow_m^B b_2$, and $c_1\rightarrow_m^B c_2$, in which case Lemma 27 of \cite{BulatovGraph1} shows that there exists some ternary term $t(x,y,z)$ such that $t(a_1,a_2,a_2)=a_2$, $t(b_2,b_1,b_2)=b_2$, and $t(c_2,c_2,c_1)=c_2$. Therefore,
\[
t\left(
\left[\begin{array}{c}
a_1\\
b_2\\
c_2
\end{array}\right],
\left[\begin{array}{c}
a_2\\
b_1\\
c_2
\end{array}\right],
\left[\begin{array}{c}
a_2\\
b_2\\
c_1
\end{array}\right]
\right)=\left[
\begin{array}{c}
a_2\\
b_2\\
c_2
\end{array}
\right].
\]
(In the above computation, we have denoted the elements of $S$ as vector columns.) In this case we also obtain $(a_2,b_2,c_2)\in R$, so Relational Axiom 3 is proved.
\end{proof}

We finish this subsection with a result which shows that $\rightarrow_a^B$ and $\rightarrow_m^B$ depend on all of $\vr v$.

\begin{prp}\label{Bedgesdepend}
Let $\vr v$ be a finitely generated pseudovariety of minimal Taylor algebras such that $\m a\in \vr v$ is a two-element semilattice, $A=\{a,b\}$ and $\{b\}\lhd_2 \{a,b\}$. Then $a\rightarrow_a^B b$ iff there exists an affine algebra in $\vr v$, while $a\rightarrow_m^B b$ iff there exists a two-element majority algebra in $\vr v$.
\end{prp}

\begin{proof}
Suppose first that there is no 2-element majority algebra in $\vr v$. Since $\vr v$ is closed under taking subalgebras and homomorphic images, the assumptions of the majority condition are never met, so that condition is vacuously true for any projection operation, e.g. $t(x,y,z)=x$ satisfies the majority condition. But $b\notin\{a\}=\Sg(a,t(a,b,b))$, so $a\rightarrow_m^B b$ does not hold. 

Similarly, if there is no affine algebra in $\vr v$, we conclude that the first projection satisfies the minority condition (though not condition (a) from Theorem~\ref{Bunified}). Hence we can take the first projection to be the operation $t(x,y,z)$ from the definition of the minority edges. But again, $b\notin\{a\}=\Sg(a,t(a,a,b))$.

On the other hand, if there is a 2-element majority algebra $\m b\in \vr v$, then consider any term $t(x,y,z)$ which satisfies the majority condition. Let $\m f=\m f_{\vr v}(x,y)$ be the free algebra, let $\varphi:\m f\rightarrow \m a$ be the homomorphism such that $\varphi(x)=a$ and $\varphi(y)=b$, and let $\psi:\m f\rightarrow\m b$ be some homomorphism such that $\psi(\{x,y\})=B$. Denote by $\alpha:=\ker\varphi$ and $\beta:=\ker\psi$. Since 
\[
\begin{gathered}
\mbox{}[t(a,a,b)]_\beta=[t(a,b,a)]_\beta=[t(b,a,a)]_\beta\\
\neq [t(a,b,b)]_\beta=[t(b,a,b)]_\beta=[t(b,b,a)]_\beta,
\end{gathered}
\]
we obtain that $t^{\sm f}$ depends on all three variables. On the other hand, from $\{b\}\lhd_2 \{a,b\}$ we obtain $[y]_\alpha\lhd_2 \m f$, so by Theorem~\ref{2absstrproj} $[y]_\alpha$ is a strongly projective subuniverse of $\m f$. Hence, $$[t(x,y,y)]_\alpha=[t(y,x,y)]_\alpha=[t(y,y,x)]_\alpha=[y]_\alpha,$$
and after applying $\varphi$, we obtain $t(a,b,b)=t(b,a,b)=t(b,b,a)=b$, and therefore $a\rightarrow_m^B b$.

The proof that, if there exists an affine algebra $\m c\in \vr v$, then $a\rightarrow_a^B b$ is similar (just take some 2-generated subalgebra of $\m c$, instead of the whole $\m c$, and then proceed as in the previous paragraph).
\end{proof}

\subsection{Our edges satisfy Edge Axioms}\label{subsec:52}

Next we define a pair of directed graphs, $as(\m a)$ and $sm(\m a)$, on each algebra $\m a$ in the finitely generated pseudovariety of minimal Taylor algebras $\vr v$ in a different way, and prove that these graphs also satisfy the Edge Axioms. These new digraphs depend only on the algebra on which they act, not on the rest of $\vr v$, and we believe their definition is easier and more natural than Bulatov's graphs. 

\begin{df}\label{ouredgesdef}
Let $\vr v$ be a finitely generated pseudovariety of minimal Taylor algebras, $\m a\in\vr v$, and  $a,b \in A$.
\begin{enumerate}
\item $a\rightarrow_{as}^Z b$ in $\m a$ if for every cyclic term $c(x_1, \dots , x_n)$ of $\vr v$, \[b\in \Sg(a,c(a,a, ... , a, b)).\]

\item  $a\rightarrow_{sm}^Z b$ in $\m a$ if for every binary term $t(x,y)$, every cyclic term $c(x_1, ... , x_n)$ of $\vr v$ and every $k>0$ such that $k\leq n/2$,
\[b \in \Sg(a, t( b, c(\underbrace{a,a,\dots,a}\limits_k,\underbrace{b,b,\dots,b}\limits_{n-k}) ) ).\] 
(In the expression $c(a,a,\dots,a,b,b,\dots,b)$, the first $k$ variables are evaluated as $a$, while the remaining $n-k$ variables are evaluated as $b$.) \end{enumerate}
\end{df}

{\noindent\bf Notation.} For shorter notation, following \cite{MM}, we will use the expression $c(a^kb^{n-k})$ instead of $c(\underbrace{a,a,\dots,a}\limits_k,\underbrace{b,b,\dots,b}\limits_{n-k})$.

We define the {\em full composition} of terms $s(x_1,\dots,x_m)$ and $t(x_1,\dots,x_n)$, denoted by $s\lhd t$, as the term
\[
\begin{gathered}
(s\lhd t)(x_1,\dots,x_{mn}):=\\
s(t(x_1,\dots,x_n),t(x_{n+1},\dots,x_{2n}),\dots,t(x_{(m-1)n+1},\dots,x_{mn})).
\end{gathered}
\]

\begin{thm}\label{ouredgesEAx}
Let $\vr v$ be a finitely generated pseudovariety of minimal Taylor algebras and for each $\m a\in\vr v$ let $as(\m a)=(A;\rightarrow_{as}^Z)$ and $sm(\m a)=(A;\rightarrow_{sm}^Z)$ be the digraphs defined in Definition~\ref{ouredgesdef}. Then $\{(\m a,as(\m a),sm(\m a)):\m a\in\vr v\}$ satisfies the Edge Axioms.
\end{thm}

\begin{proof}
{\bf Base Axiom 1.} Let $\m a\in\vr v$, let $\m B$ be an affine subalgebra of $\m A$ and let $c$ be a cyclic term of $\vr v$. If $\m a$ is the full idempotent term reduct of some faithful $\m r$-module, then we know \[c(x_1,\dots,x_n)=\sum_{i=1}^n\alpha x_i\] for some $\alpha\in R$ such that $n\alpha=1$ in $\m r$. If we iterate the full composition of $c$ with itself, i.e., if $c_1:=c$ and $c_{k+1}:=c\lhd c_k$, inductively we obtain
\[c_k(x_1,\dots,x_{n^k})=\sum_{i=1}^{n^k}\alpha^k x_i,\]
where we know $n^k\alpha^k=1^k=1$.

Now, $\alpha$ is an endomorphism of the finite Abelian group $(A;+)$, so, just like any mapping of a finite set into itself, there exists an idempotent power (for composition) of $\alpha$, i.e., there exists some $m$ such that $\alpha^{2m}=\alpha^m$. Let us compute 
\begin{align*}
c_m(a,\dots,a,b)&=c_m(c_m(a,\dots,a,b),c_m(a,\dots,a,b),\dots,c_m(a,\dots,a,b))\\ &=n^m\alpha^m(1-\alpha^m) a + n^m \alpha^{2m} b\\
& = n^m(\alpha^m-\alpha^m) a + n^m\alpha^m b =  b.
\end{align*}
Moreover, for each $k$ we have that 
\[
c_{k+1}(a,\dots,a,b)=c(a,\dots,a,c_k(a,\dots,a,b))\in \Sg(a,c_k(a,\dots,a,b)),
\]
so, inductively, each $c_k(a,\dots,a,b)\in \Sg(a,c(a,\dots,a,b))$. We obtain, as desired, $b=c_m(a,\dots,a,b)\in \Sg(a,c(a,\dots,a,b))$, and so $a\rightarrow_{as}^Z b$.

{\bf Base Axiom 2.} Let $\m a\in\vr v$, let $\m B$ be a majority subalgebra of $\m a$ and $a,b\in B$. Since $\m a$ is a minimal Taylor algebra and there is a ternary term which acts as the majority on $\{a,b\}$, we know that $\Sg(a,b) = \{a,b\}$, so we may as well assume $B = \{a,b\}$. Let $c(x_1,\dots,x_n)$ be a cyclic term on $\vr v$ and $k\leq n/2$. Consider 
\[p(x,y)=c(x^ky^{n-k}).\] (We're using the notation we established just after Definition~\ref{ouredgesdef}). As the two-element majority algebra only has the trivial binary operations, $p(x,y)$ it is either equal to $x$ or to $y$ on $\{a,b\}$. 

If $p(x,y)=x$, by the cyclic equations, 
\[c(y^{n-k}x^k)=x,\] 
which implies that $k\neq n-k$ (that is $k<n/2<n-k$). Moreover, we would get 
\begin{align*}
c(x^ky^{n-k})&=x,\\
c(x^{n-k}y^k)&=y\text{ and}\\
c(x^n)&=x
\end{align*}
which, together with $k<n-k<n$ contradicts the well-known fact that the term operations on the two-element majority algebra are monotone with respect to order $a<b$.

Hence, $p(x,y)=y$. Again, by the cyclic equations, 
\[c(y^{n-k}x^k)=y,\] 
which implies that $k\neq n-k$ (that is $k<n-k$). 
Therefore, \[c(a^kb^{n-k}) = b.\] To prove the desired conclusion $a\rightarrow_{sm}^Z b$, we need to prove for any binary term $t(x,y)$ that \[b \in \Sg(a,t(b,c(a^kb^{n-k}))) = \Sg(a,t(b,b))=\Sg(a,b),\] 
which is trivially true.

{\bf Base Axiom 3.} If $\m a\in\vr v$ and $\{a,b\}\in \Sub a$ is a two-element meet-semilattice with absorbing element $b$ and neutral element $a$, then any cyclic term $c$ of $\vr v$, when interpreted in the semilattice $\{a,b\}$, must have every variable essential. Hence $c(x_1,\dots,x_n)$ equals on $\{a,b\}$ to the meet of all its variables. So $c(a,a,\dots,a,b) = b$ and $a\rightarrow_{as}^Z b$. Also, for any binary term $t(x,y)$ and $k\leq n/2$, \[t(b,c(a^kb^{n-k})) = t(b,b) = b,\] 
so 
\[b \in \Sg(a,b) = \Sg(a, t( b, c(a^kb^{n-k}) ) ).\]
Thus $a\rightarrow_{sm}^Z b$, too, so $a\rightarrow_{s}^Z b$, as required.

{\bf Homomorphism Axiom 1.} Let $\m a,\m b\in\vr v$, let $\varphi:\m a\rightarrow\m b$ be a homomorphism, and let $a\rightarrow_{as}^Z b$ or $a\rightarrow_{sm}^Z b$. If $c(x_1,\dots,x_n)$ is any cyclic term of $\vr v$, $k\leq n/2$ and $t(x,y)$ is any binary term, we have that either $b\in \Sg^{\m a}(a,c(a^{n-1}b))$, or $b\in\Sg^{\m a}(a,t(b,c(a^kb^{n-k})))$. In the first case this implies $\varphi(b)\in \Sg^{\m b}(\varphi(a),c(\varphi(a)^{n-1}\varphi(b)))$, while in the second case this implies $\varphi(b)\in\Sg^{\m b}(\varphi(a),t(\varphi(b),c(\varphi(a)^k\varphi(b)^{n-k})))$. Homomorphism Axiom 1 follows directly.

{\bf Homomorphism Axiom 2.} Let $\m a,\m b\in\vr v$, let $\varphi:\m a\rightarrow\m b$ be a homomorphism, and let $\varphi(a)\rightarrow_{as}^Z \varphi(b)$ or $\varphi(a)\rightarrow_{sm}^Z \varphi(b)$. Select $b'\in\varphi^{-1}(\varphi(b))$ such that $\Sg(a,b')$ is minimal (with respect to inclusion) among the sets $\{\Sg(a,c):\varphi(c)=\varphi(b)\}$. 

If $\varphi(a)\rightarrow_{as}^Z \varphi(b)$ and $c(x_1,\dots,x_n)$ is a cyclic term of $\vr v$, then there exists a binary term $p(x,y)$ such that
\[p^{\m b}(\varphi(a),c^{\m b}( \varphi(a)^{n-1}\varphi(b) ))=\varphi(b).\]
Now we denote $b'':=p^{\m a}(a,c^{\m a}(a^{n-1}b'))$ and compute
\begin{align*}
\varphi(b'')&=\varphi(p^{\m a}(a,c^{\m a}(a^{n-1}b')))\\
&=p^{\m b}(\varphi(a),c^{\m b}(\varphi(a)^{n-1}\varphi(b')))\\
&=p^{\m b}(\varphi(a),c^{\m b}( \varphi(a)^{n-1}\varphi(b) ))=\varphi(b).
\end{align*}

By its definition, $b''\in\Sg(a,b')$ and also $b''\in \Sg(a,c^{\m a}(a^{n-1}b'))$, so $\Sg(a,b'')\subseteq\Sg(a,b')$. On the other hand, from $\varphi(b'')=\varphi(b)$, by the choice of $b'$ it can't happen that $\Sg(a,b'')\subsetneq\Sg(a,b')$, so $b'\in\Sg(a,b'')$. Hence, $b'\in\Sg(a,c^{\m a}(a^{n-1}b'))$ and thus $a\rightarrow_{as}^Z b'$.

On the other hand, if $\varphi(a)\rightarrow_{sm}^Z \varphi(b)$, $c(x_1,\dots,x_n)$ is a cyclic term of $\vr v$, $k\leq n/2$ and $t(x,y)$ is any binary term, then there exists a binary term $p(x,y)$ such that 
\[p^{\m b}(\varphi(a),t^{\m b}(\varphi(b),c^{\m b}( \varphi(a)^{k}\varphi(b)^{n-k}) ))=\varphi(b).\]
Denote $b'':=p^{\m a}(a,t^{\m a}(b',c^{\m a}(a^k(b')^{n-k}) ))$. Again we note that $b''\in\Sg(a,b')$ and also $b''\in \Sg(a,t^{\m a}(b',c^{\m a}(a^k(b')^{n-k})))$. We compute
\begin{align*}
\varphi(b'')&= \varphi(p^{\m a}(a,t^{\m a}(b',c^{\m a}(a^k(b')^{n-k}) )))\\
&= p^{\m b}(\varphi(a),t^{\m b}(\varphi(b'),c^{\m b}(\varphi(a)^k\varphi(b')^{n-k})))\\
&= p^{\m b}(\varphi(a),t^{\m b}(\varphi(b),c^{\m b}(\varphi(a)^k\varphi(b)^{n-k})))=\varphi(b).
\end{align*}

Since $b''\in \Sg(a,b')$, we obtain $\Sg(a,b'')\subseteq\Sg(a,b')$. On the other hand, from $\varphi(b'')=\varphi(b)$, by the choice of $b'$ it can't happen that $\Sg(a,b'')\subsetneq\Sg(a,b')$, so $b'\in\Sg(a,b'')$. Moreover, since $b''\in \Sg(a,t^{\m a}(b',c^{\m a}(a^k(b')^{n-k})))$, it follows that $b'\in \Sg(a,t^{\m a}(b',c^{\m a}(a^k(b')^{n-k})))$ holds, so by definition $a\rightarrow_{sm}^Z b'$.

{\bf Relational Axiom 1.} Suppose that $\m a,\m b\in\vr v$, $\m r\leq \m a\times\m b$, $a_1\rightarrow_{as}^Z a_2$ in $\m a$, $b_1\rightarrow_{sm}^Z b_2$ in $\m b$ and $(a_1,b_2),(a_2,b_1)\in R$. Let $c(x_1,\dots,x_n)$ be a cyclic term of $\vr v$. Since $a_2\in \Sg^{\m a}(a_1,c(a_1^{n-1}a_2))$, there exists a term $t(x,y)$ such that $t(a_1,c(a_1^{n-1}a_2))=a_2$. Applying this term in $\m r$ we compute
\[
t\left(\begin{bmatrix}
a_1\\
b_2
\end{bmatrix}
,c\left(
\begin{bmatrix}
a_1\\
b_2
\end{bmatrix}
^{n-1}
\begin{bmatrix}
a_2\\
b_1
\end{bmatrix}
\right)
\right)
=
\begin{bmatrix}
a_2\\
t(b_2,c(b_2^{n-1}b_1))
\end{bmatrix}.
\]
So, we conclude that $b_1,t(b_2,c(b_2^{n-1}b_1))\in [a_2]R$. By the definition of sm-edges, from $b_1\rightarrow_{sm}^Z b_2$ we have that $b_2\in \Sg^{\m b}(b_1,t(b_2,c(b_2^{n-1}b_1)))$. Since $[a_2]R$ is a subuniverse of $\m b$ which contains both $b_1$ and $t(b_2,c(b_2^{n-1}b_1))$, we obtain that $b_2\in [a_2]R$, i.e., $(a_2,b_2)\in R$, as desired.

{\bf Relational Axiom 2.} Suppose that $\m a,\m b\in\vr v$, $\m r\leq \m a\times\m b$, $a_1\rightarrow_{as}^Z a_2$ in $\m a$, $b_1\rightarrow_{as}^Z b_2$ in $\m b$, $(a_1,b_1),(a_1,b_2),(a_2,b_1)\in R$ and let $c(x_1,\dots,x_n)$ be a cyclic term for $\vr v$. Then there exists a term $t(x,y)$ so that $t(a_1,c(a_1^{n-1}a_2)) = a_2$. We compute in $\m r$:
\begin{align*}
	t\left(
	c\left(\begin{bmatrix}
		a_1\\
		b_2
	\end{bmatrix}
	\begin{bmatrix}
		a_1\\
		b_1
	\end{bmatrix}^{n-1}\right),
	c\left(
	\begin{bmatrix}
		a_1\\
		b_2
	\end{bmatrix}
	\begin{bmatrix}
		a_1\\
		b_1
	\end{bmatrix}^{n-2}
	\begin{bmatrix}
		a_2\\
		b_1
	\end{bmatrix}
	\right)
	\right)
	=&
	\begin{bmatrix}
		t(a_1,c(a_1^{n-1}a_2)\\
		t(c(b_2b_1^{n-1}),c(b_2b_1^{n-1}))
	\end{bmatrix}\\
	=&
	\begin{bmatrix}
		a_2\\
		c(b_2b_1^{n-1})
	\end{bmatrix}.
\end{align*}

So, $c(b_1^{n-1}b_2)=c(b_2b_1^{n-1})\in [a_2]R$ and we already know that also $b_1\in [a_2]R$. By the definition of as-edges, from $b_1\rightarrow_{as}^Z b_2$ we have that $b_2\in \Sg^{\m b}(b_1,c(b_1^{n-1}b_2))$. Since $[a_2]R$ is a subuniverse of $\m b$ which contains both $b_1$ and $c(b_1^{n-1}b_2)$, we obtain that $b_2\in [a_2]R$, i.e., $(a_2,b_2)\in R$, as desired.

{\bf Relational Axiom 3.} Suppose that $\m a,\m b, \m c\in\vr v$, $\m r\leq \m a\times\m b\times \m c$, $a_1\rightarrow_{sm}^Z a_2$ in $\m a$, $b_1\rightarrow_{sm}^Z b_2$ in $\m b$ and $c_1\rightarrow_{sm}^Z c_2$ in $\m c$. Moreover, suppose that $(a_1,b_2,c_2),$ $(a_2,b_1,c_2),$ $(a_2,b_2,c_1)\in R$, that $c(x_1,\dots,x_n)$ is a cyclic term for $\vr v$, where $n\geq 5$ and that $n/3\leq k< n/2$. Consequently, $0<n-2k< n/2$. Since $a_1\rightarrow_{sm}^Z a_2$, there exists a binary term $t_1(x,y)$ such that $t_1(a_1,c(a_1^ka_2^{n-k}))=a_2$. We compute in $\m r$:
\begin{align*}
t_1\left(
\begin{bmatrix}
	a_1\\
	b_2\\
	c_2
\end{bmatrix}
,
c\left(
\begin{bmatrix}
	a_1\\
	b_2\\
	c_2
\end{bmatrix}^k
\begin{bmatrix}
	a_2\\
	b_1\\
	c_2
\end{bmatrix}^k
\begin{bmatrix}
	a_2\\
	b_2\\
	c_1
\end{bmatrix}^{n-2k}
\right)
\right)
=&
\begin{bmatrix}
	t_1(a_1,c(a_1^ka_2^{n-k})\\
	t_1(b_2,c(b_2^kb_1^kb_2^{n-2k}))\\
	t_1(c_2,c(c_2^{2k}c_1^{n-2k}))
\end{bmatrix}\\
=&
\begin{bmatrix}
	a_2\\
	t_1(b_2,c(b_1^kb_2^{n-k}))\\
	t_1(c_2,c(c_1^{n-2k}c_2^{2k}))
\end{bmatrix}
\in R.
\end{align*}

We have $b_1\rightarrow_{sm}^Z b_2$, from which we have that there exists some binary term $t_2(x,y)$ such that $t_2(b_1,t_1(b_2,c(b_1^kb_2^{n-k})))=b_2$. Again we compute in $\m r$:
\begin{align*}
t_2\left(
\begin{bmatrix}
	a_2\\
	b_1\\
	c_2
\end{bmatrix}
,
\begin{bmatrix}
	a_2\\
	t_1(b_2,c(b_1^kb_2^{n-k}))\\
	t_1(c_2,c(c_1^{n-2k}c_2^{2k}))
\end{bmatrix}
\right)
=&
\begin{bmatrix}
	a_2\\
	t_2(b_1,t_1(b_2,c(b_1^kb_2^{n-k})))\\
	t_2(c_2,t_1(c_2,c(c_1^{n-2k}c_2^{2k})))
\end{bmatrix}\\
=&
\begin{bmatrix}
	a_2\\
	b_2\\
	t_2(c_2,t_1(c_2,c(c_1^{n-2k}c_2^{2k})))
\end{bmatrix}\in R.
\end{align*}

Now denote by $t_3(x,y):=t_2(x,t_1(x,y))$. The above computation yields to the conclusion that $(a_2,b_2,t_3(c_2,c(c_1^{n-2k}c_2^{2k})))\in R$. Since $c_1\rightarrow_{sm}^Z c_2$ and $n-2k<n/2$, we obtain that $c_2\in\Sg(c_1,t_3(c_2,c(c_1^{n-2k}c_2^{2k})))$. We also know, by idempotence, that $S:=\{x\in C:(a_2,b_2,x)\in R\}$ is a subuniverse of $C$ and that $c_1,t_3(c_2,c(c_1^{n-2k}c_2^{2k}))\in S$, and hence $c_2\in S$. That means $(a_2,b_2,c_2)\in R$, as desired.
\end{proof}

\section{Basic consequences of Edge Axioms}

First we prove some basic consequences of the edge axioms. To shorten the formulations of theorems, for the rest of this section, we assume that $\vr v$ is a finitely generated pseudovariety of minimal Taylor algebras and that for each $\m a\in \vr v$ there are two directed graphs fixed, $as(\m a)=(A,\rightarrow_{as})$ and $sm(\m a)=(A,\rightarrow_{sm})$, so that the class $\{(\m a,as(\m a),sm(\m a):\m a\in\vr v)\}$ satisfies the Edge Axioms. (We note here that we only need $\m f_{\vr v}(2)$ to be finite, but in all applications of our theory we will have a finitely generated $\vr v$). We denote $\rightarrow_s:=\rightarrow_{as}\,\cap\,\rightarrow_{sm}$ and $\rightarrow_{asm}:=\rightarrow_{as}\,\cup\,\rightarrow_{sm}$, as before.

\begin{lm}\label{existsaterm}
Let $\m a\in \vr v$ and $a\rightarrow_{asm} b$ in $\m a$. Then there exists a term $t(x_1,\dots,x_n)$ such that $x_1$ and $x_n$ are essential variables of $t^{\sm a}$ and that $t^{\sm a}(a,c_2,\dots,c_{n-1},b)=b$ for some choice of $c_i\in A$.
\end{lm}

\begin{proof}
If $a=b$ there is nothing to prove (use any Taylor operation of $\m a$, it is idempotent and has at least two essential variables), so we suppose that $a\neq b$.

Suppose first that $a\rightarrow_{as}b$ holds. From Relational Axiom 2 follows that $(b,b)\in\Sg^{\sm a^2}(\{(a,b),(a,a),(b,a)\})$. Therefore, there exists a ternary term $t$ such that $t^{\sm a}(b,a,a)=t^{\sm a}(a,a,b)=b$. By idempotence it follows that $t^{\sm a}(a,a,a)=a$, and hence the first and the third positions are both essential in the term operation $t^{\sm a}$.

Next, suppose that $a\rightarrow_{sm}b$. Then $(b,b,b)\in\Sg^{\sm a^3}(\{(a,b,b),(b,a,b),(b,b,a)\})$ follows from Relational Axiom 3. Therefore, there exists a ternary term $t(x,y,z)$ such that $t^{\sm a}(a,b,b)=t^{\sm a}(b,a,b)=t^{\sm a}(b,b,a)=b$. By idempotence, $t^{\sm a}(a,a,a)=a$, hence $t^{\sm a}(a,b,a)$ can not be simultaneously equal to $t^{\sm a}(b,b,a)=b$ and to $t^{\sm a}(a,a,a)=a$. Thus, at least one of the first two positions of $t(x,y,z)$ is essential in $t^{\sm a}$, and the same argument shows that in any pair of variables of $t(x,y,z)$ at least one is essential in $t^{\sm a}$. It follows that $t^{\sm a}$ has at least two essential variables, and thus $t^{\sm a}(a,b,b)=t^{\sm a}(b,a,b)=t^{\sm a}(b,b,a)=b$ finishes the lemma.
\end{proof}

Lemma~\ref{existsaterm} has an easy, but useful corollary:

\begin{cor}\label{noedge}
Let $\m a\in\vr v$ and $a\rightarrow_s b$ in $\m a$. Then $b\not\rightarrow_{asm} a$.
\end{cor}

\begin{proof}
Applying Lemma~\ref{therecanbonly1sedge} and Corollary~\ref{sedge2abs} we obtain that $B=\{a,b\}$ is a subuniverse of $\m a$, while $\{b\}$ is a strongly projective subuniverse of $\m b$. Suppose that $b\rightarrow_{asm} a$ in $\m a$, by Proposition~\ref{sameinsubalg} we obtain that $b\rightarrow_{asm} a$ in $\m b$. By Lemma~\ref{existsaterm}, there exists a term $t$ such that $t^{\sm b}(b,?,?,\dots,?)=a$, where the first variable is essential in $t^{\sm b}$. This contradicts the fact that $\{b\}$ is a strongly projective subuniverse of $\m b$.
\end{proof}

Next we consider the following stronger variants of the Base Axioms:

\begin{bax1}
Let $\m a\in\vr v$, and let $\m b\leq \m a$ be an affine subalgebra. Then for all $a,b\in B$, $a\neq b$, $a\rightarrow_{as} b$ and $a\not\rightarrow_{sm} b$ hold.
\end{bax1}

\begin{bax1}
Let $\m a\in \vr v$, and let $\m b\leq \m a$ be a subalgebra with a majority term. Then for all $a,b\in B$, $a\neq b$, $a\rightarrow_{sm} b$ and $a\not\rightarrow_{as} b$ hold.
\end{bax1}

\begin{bax1}
Let $\m a\in \vr v$, and let $a\neq b$ in $A$. There exist a term $t(x,y)$ such that $t$ acts on the set $\{a,b\}$ as a semilattice operation with the absorbing element $b$ iff $a\rightarrow_s b$. If $a\rightarrow_s b$, then $b\not\rightarrow_{asm} a$.
\end{bax1}

\begin{prp}
Stronger Base Axioms 1-3 are also satisfied in $\vr v$.
\end{prp}

\begin{proof}
{\bf Stronger Base Axiom 3}: The first statement is Lemma~\ref{therecanbonly1sedge}, while the second statement is Corollary~\ref{noedge}.

{\bf Stronger Base Axioms 1 and 2}: Base Axiom 1 guarantees $a\rightarrow_{as} b$ for all $a,b\in B$. Assume that there exist $a,b\in B$ such that $a\rightarrow_{sm} b$, and therefore $a\rightarrow_s b$. Stronger Base Axiom 3 implies that $b\not\rightarrow_{as} a$, contradicting Base Axiom 1. The proof of Stronger Base Axiom 2 is analogous to the proof of Stronger Base Axiom 1, we just flip as and sm edges and use Base Axiom 2 instead of Base Axiom 1.
\end{proof}

The next proposition is a part of A. Bulatov's Proposition 24 from \cite{BulatovGraph1}. However, we give the proof from \cite{Zebnotes}, since it is both shorter and easier.

\begin{prp}\label{universalmeet}
There exists a binary term $f(x,y)$ such that $\vr v\models f(x,y)\approx  f(x,f(x,y))$ and for any $\m a\in \vr v$ and any $a,b\in A$, either $f(a,b)=a$ or $a\rightarrow_s f(a,b)$. Moreover, whenever $a\rightarrow_s b$, then $f(a,b)=f(b,a)=b$.
\end{prp}

\begin{proof}
Take any cyclic term of $\vr v$ and make all of its variables, except the first one, equal to obtain the term $t(x,y)$. For all $\m a\in\vr v$ and any $a,b\in A$ such that $a\rightarrow_s b$, Corollary \ref{sedge2abs} implies that $B=\{a,b\}$ is a subuniverse of $\m a$ and $\{b\}$ is a strongly projective subuniverse of $\m b$. Moreover, each variable of a cyclic term is essential (also in the subalgebra $\m b$), so $t(a,b)=t(b,a)=b$ for every $a\rightarrow_s b$ in $\m a$. Every term we will construct in this proof will be a two-variable term in the language $\{t\}$ in which the variables $x$ and $y$ will both occur, so the basic properties of semilattices imply that for all $\m a\in\vr v$, any $a,b\in A$ such that $a\rightarrow_s b$ and all terms, those terms act on $\{a,b\}$ as the semilattice operation with the absorbing element $b$.

Now we iterate the term $t$ in the first variable: 
\[t_1(x,y):= t(x,y)\text{ and }t_{i+1}(x,y):=t(x,t_i(x,y)).\] 
It is easy to compute that for $n:=|F_{\vr v}(2)|$ and $k=n!$ it holds that $\vr v\models t_k(x,t_k(x,y))\approx t_k(x,y)$.

Next, we denote $q(x,y):=t_k(x,t_k(y,x))$ and compute
\[
\vr v\models t_k(x,q(x,y))= t_k(x,t_k(x,t_k(y,x)))\approx t_k(x,t_k(y,x))=q(x,y),\]
and thus
\[
\begin{gathered}
\vr v\models q(q(x,y),x)= t_k(t_k(x,t_k(y,x)),t_k(x,t_k(x,t_k(y,x))))\approx\\
t_k(t_k(x,t_k(y,x)),t_k(x,t_k(y,x)))\approx t_k(x,t_k(y,x)) = q(x,y).
\end{gathered}
\]

Finally, we use another iteration in the first variable 
\[q_1(x,y):=q(x,y)\text{ and }q_{i+1}(x,y):=q(x,q_i(x,y)).\]
As before for $t_k$, $$\vr v\models q_k(x,q_k(x,y))\approx q_k(x,y).$$ 

We prove by an induction on $i$ that for all $i$, $\vr v\models q_i(q(x,y),x)\approx q(x,y)$. We proved that it holds for $q_1=q$, and suppose that it holds for $q_i$. Then 
\[\vr v\models q_{i+1}(q(x,y),x)= q(q(x,y),q_i(q(x,y),x))\approx q(q(x,y),q(x,y))\approx q(x,y).\]
By substituting $q_{k-1}(x,y)$ for $y$ in the identity $q_k(q(x,y),x)\approx q(x,y)$ we obtain
\[\vr v\models q_k(q(x,q_{k-1}(x,y)),x)\approx q(x,q_{k-1}(x,y)),\]
which can be written differently as
\[\vr v\models q_k(q_k(x,y),x)\approx q_k(x,y).\]

We choose for $f(x,y):=q_k(x,y)$. The identities $f(x,f(x,y))\approx f(x,y)\approx f(f(x,y),x)$ imply that for all $\m a \in \vr v$ and all $a,b\in A$, either $f(a,b)=a$, or $a\rightarrow_s f(a,b)$ in $\m a$. Moreover, we know that for all $\m a \in \vr v$ and all $a,b\in A$ such that $a\rightarrow_s b$ in $\m a$, $f(a,b)=f(b,a)=b$, so the proposition is proved.
\end{proof}

\section{Weak connectivity and binary absorption}

In this section, we apply our edges to derive simpler proofs of three major results from \cite{BulatovGraph1}, \cite{dreamteam} and \cite{BulatovGraph2}.

The next theorem and its proof are our version of Theorem 5 of \cite{BulatovGraph1}. Its proof is shorter than in \cite{BulatovGraph1}, and we use much less theory: in \cite{BulatovGraph1}, A. Bulatov uses the Tame Congruence Theory of D. Hobby and R. McKenzie \cite{tct} as well as three deep papers by K. Kearnes, by \'A. Szendrei and by M. Valeriote, while we need just Theorems~\ref{Taylor=cyclic} and \ref{affineeq} of the major theoretic results. Our proof first appeared in the online notes \cite{Zebnotes} by the first author, and is also similar to the proof of Theorem 3.2 in \cite{dreamteam}.

\begin{thm}\label{weakconnect}
Let $\m a\in \vr v$. Then $asm(\m a)$ is weakly connected.
\end{thm}

\begin{proof}
We use an induction on $|A|$. When $|A|=2$, the Post lattice reveals that there are only four minimal clones with a Taylor term, the two semilattice clones, the affine clone and the clone generated by the majority operation. Moreover, any Taylor clone contains one of these four, so these are all minimal Taylor clones. These clones have weakly connected asm-graphs by Base Axioms 1-3.

For the inductive step, let $a,b\in A$ be two distinct elements. We have to prove that $a$ is weakly connected to $b$. First note that if $B=\Sg(a,b)\neq A$ then by the inductive assumption $a$ is weakly connected to $b$ in $asm(\m b)$, and then we can apply Proposition~\ref{sameinsubalg}. Moreover, we easily prove the following:

\begin{claim}
If $a$ and $b$ are connected in the hypergraph $H(\m a)$, then $a$ and $b$ are weakly connected in $asm(\m a)$.
\end{claim}
\begin{proof}[Proof of the claim]
The assumption implies that there exists a sequence $C_1,\dots,C_k$ of proper subuniverses of $A$ which connect $a$ to $b$. In other words, $a\in C_1$, $b\in C_k$ and for all $i<k$, $C_i\cap C_{i+1}\neq\emptyset$. Select $a=c_0,c_1,\dots,c_k=b$ so that for each $1\leq i<k$, $c_i\in C_i\cap C_{i+1}$. For each $i<k$, by the inductive assumption, $c_i$ and $c_{i+1}$ are weakly connected in $asm(\m c_{i+1})$. By Proposition~\ref{sameinsubalg}, $c_i$ and $c_{i+1}$ are also weakly connected in $asm(\m a)$, and the concatenation connects $a$ and $b$.
\end{proof}

\begin{claim}
If $\m a$ has a nontrivial congruence $\alpha$, then $a$ and $b$ are weakly connected in $asm(\m a)$.
\end{claim}
\begin{proof}[Proof of the claim]
By the inductive assumption, $[a]_\alpha$ and $[b]_{\alpha}$ are weakly connected in $asm(\m a/\alpha)$. This means that we can select some $\alpha$-blocks $[a]_\alpha=[c_0]_\alpha, [c_1]_\alpha,\dots,[c_k]_\alpha=[b]_\alpha$ such that for each $i<k$, either $[c_i]_\alpha\rightarrow_{asm}[c_{i_1}]_\alpha$, or $[c_{i_1}]_\alpha\rightarrow_{asm}[c_i]_\alpha$. Now, using the natural homomorphism from $\m a$ onto $\m a/\alpha$ and Homomorphism Axiom 2, we can select $d_i$ and $e_i$, $i\leq k$, so that $d_i,e_i\in [c_i]_\alpha$, $a=d_0$, $b=e_k$ and for all $i<k$ we have that at least one of $e_i\rightarrow_{asm}d_{i+1}$ and $d_{i+1}\rightarrow_{asm}e_i$ holds. Moreover, since each $[c_i]_\alpha$ is a subuniverse of $\m a$ which is strictly smaller, thus for all $i\leq k$, by the inductive hypothesis we know that $d_i$ and $e_i$ are weakly connected in $G_{asm}(\m a)$. Therefore, $a$ and $b$ are weakly connected, as desired.
\end{proof}

\begin{claim}
If $\m a$ has a nontrivial tolerance $\tau$, then $a$ and $b$ are weakly connected in $asm(\m a)$. 
\end{claim} 
\begin{proof}[Proof of the claim]
 By Claim 2, we can assume that $\m a$ is simple. Since the transitive closure of $\tau$ is a congruence, either this closure is $0_{\m a}$, in which case $\tau=0_{\m a}$, or this closure is $1_{\m a}$, in which case $\tau$ is connected. By Proposition~\ref{conntolH}, if $\tau$ is connected and $\tau\neq 1_{\m a}$, then $H(\m a)$ is connected, so $a$ and $b$ are weakly connected by Claim 1.
\end{proof}

\begin{claim}
 For any $\m b$ and any $\m r\leq_{sd}\m a\times\m b$, we can assume that either $R$ is the graph of a surjective homomorphism from $\m b$ onto $\m a$, or there exists some $c\in B$ such that $R[c]=A$. 
\end{claim}
\begin{proof}[Proof of the claim]
For the congruence $lk_1R$, by Claim 2 the only options are $lk_1R=0_{\m a}$ and $lk_1R=1_{\m a}$. Assume that $lk_1R=1_{\m a}$. Using the assumption that $A=\Sg(a,b)$ and according to Theorem~\ref{connected-linked}, if $tol_1R\neq A^2$, then we are done by Claim 3. So, we can assume that $tol_1R = A^2$, using again Theorem~\ref{connected-linked} and the fact that $A=\Sg^{\sm a}(a,b)$, we obtain that there exists some $c\in B$ such that $A=R[c]$. On the other hand, if $lk_1R=0_{\m a}$, then $R$ is the graph of a surjective homomorphism from $\m b$ onto $\m a$.
\end{proof}

\begin{claim}
 For any $c,d\in A$ which are not in the same weak component of $asm(\m a)$, the relation $R_{cd}:=\Sg^{\m a^2}((c,d),(d,c))$ is the graph of an automorphism $\varphi_{cd}\in \aut a$ such that $\varphi_{cd}(c)=d$, $\varphi_{cd}(d)=c$ and $\varphi_{cd}^2=id$. 
\end{claim}
\begin{proof}[Proof of the claim]
 Since $c$ and $d$ are not weakly connected, from the inductive assumption follows that $A=\Sg^{\sm a}(c,d)$, and thus $R_{cd}\leq_{sd}\m a^2$. If $R_{cd}$ is not the graph of an automorphism, then by Claim 4 there exists some $e\in A$ such that $A=R[e]$. If $\Sg(c,e)\neq A$ and $\Sg(d,e)\neq A$, then $c$ and $d$ are connected in $H(\m a)$ and we have a contradiction by Claim 1. So assume, without loss of generality, that $\Sg(c,e)=A$. Since $(d,e),(d,c)\in R$ and using idempotence in the first coordinate, we obtain that $[d]R=A$, and hence $(d,d)\in R$. From $(d,d)\in R=\Sg^{\m a^2}((c,d),(d,c))$ we obtain that there exists some term $t(x,y)$ such that $t(c,d)=t(d,c)=d$. Together with the idempotence, we know that $t$ acts as a semilattice operation on $\{c,d\}$, so by Theorem~\ref{subuniversesminTaylor} and Base Axiom 3, $c\rightarrow_s d$. Thus, $c$ and $d$ are weakly connected, a contradiction. Hence $R_{cd}$ must be the graph of an automorphism, and the remaining claims trivially follow.
\end{proof}

\begin{claim}
 If $asm(\m a)$ is not weakly connected, then the automorphism group $\aut a$ acts transitively on $\m a$.
\end{claim}

\begin{proof}[Proof of the claim]
Consider any $c,d\in A$. If $c$ and $d$ are disconnected in $asm(\m a)$, by Claim 5 the automorphism $\varphi_{cd}$ maps $c$ to $d$. On the other hand, if $c$ and $d$ are connected in $asm(\m a)$ there must exist some $e\in A$ which is disconnected from both $c$ and $d$ in $asm(\m a)$, so the automorphism which maps $c$ to $d$ is $\varphi_{de}\circ\varphi_{ce}$.
\end{proof}

We assume from now on that $a$ and $b$ are not weakly connected in $asm(\m a)$, that $\m a=\Sg^{\sm a}(a,b)$ is tolerance-free (thus also simple), we fix $R:=R_{ab}$ and the automorphism $\varphi:=\varphi_{ab}$.

Next we consider $S:=\Sg^{\m a^3}((a,a,b),(a,b,a),(b,a,a))$. We note, as it will be used later, that $S$ is {\em symmetric}, i.e., it is invariant under permutations of coordinates, since its generating set is symmetric. Let $T(x,y)$ be the relation $T(x,y):=S(x,y,x)$.
\begin{claim}
$(a,a,a)\notin S$ and $(b,a,b)\notin S$.
\end{claim}
\begin{proof}[Proof of the claim]
If $(b,a,b)\in S$, then there exist a term $s(x,y,z)$ such that $s(b,a,a)=s(a,a,b)=b$, while $s(a,b,a)=a$. Consider the term $t(x,y,z):=s(x,s(x,y,z),z)$, it satisfies $t(b,a,a)=t(a,b,a)=t(a,a,b)=a$, and such a term $t$ is also obtained from the assumption $(a,a,a)\in S$. Applying the automorphism $\varphi$, we discover that $(\{a,b\};t)$ is the majority algebra, so by Theorem~\ref{subuniversesminTaylor} and Base Axiom 2, $a\rightarrow_{sm}b$, contradicting our assumptions.
\end{proof}

\begin{claim}
 $\mathrm{pr}_1 T\cap\mathrm{pr}_2 T\neq\emptyset$.
\end{claim}
\begin{proof}[Proof of the claim]
 Let $p$ be a prime number such that $p=3k+1$, $k\geq |A|$ and let $c(x_1,\dots,x_p)$ be a cyclic term of $\m a$. Using the notation we defined in Subsection~\ref{subsec:52}, we define the terms $$p(x,y,z):=c(x^ky^{k+1}z^k)\text{ and }q(x,y,z):=c(x^{k+1}y^{k-1}z^{k+1}).$$
Now we have
\[
p\left(
\begin{bmatrix}
	a & a & b\\
	a & b & a\\
	b & a & a
\end{bmatrix}
\right)=
\begin{bmatrix}
	c\\
	d\\
	c
\end{bmatrix}
\in S\;
\text{ and }\;
q\left(
\begin{bmatrix}
	a & a & b\\
	a & b & a\\
	b & a & a
\end{bmatrix}
\right)=
\begin{bmatrix}
	d\\
	e\\
	d
\end{bmatrix}
\in S
\]
for some $c,d\in A$, proving the Claim.
\end{proof}

\begin{claim}
 $\mathrm{pr}_1 T=A$.
\end{claim}
\begin{proof}[Proof of the claim]
Suppose first that $B=\mathrm{pr}_1 T\neq A\neq\mathrm{pr}_2 T=C$. Then $B,C\in H(\m a)$, $a\in B$, $b\in C$ and by Claim 8, $B\cap C\neq\emptyset$, in which case Claim 1 finishes the proof.

Now suppose that $B=\mathrm{pr}_1 T\neq A = \mathrm{pr}_2 T$. Since $|\mathrm{pr}_1 T|<|A|$ there can't be any homomorphisms from $\m b$ onto $\m a$, so Claim 4 implies that there exists some $c\in A$ such that $[c]T=A$. By Claim 6, there exists some automorphism $\sigma\in\aut a$ such that $\sigma(a)=c$. As $(c,\sigma(b))\in T$, it follows that $(\sigma(a),\sigma(b),\sigma(a))\in S$. As $S$ is invariant under permutations of coordinates, hence $(\sigma(a),\sigma(a),\sigma(b)),(\sigma(b),\sigma(a),\sigma(a))\in S$ and therefore, $\sigma(S)\subseteq S$. Hence $\sigma(S)=S$ as $\sigma$ is injective and $S$ is finite. That means that also $\sigma(T)=T=\sigma^{-1}(T)$ and from $\{c\}\times A\subseteq T$ follows that $\{a\}\times A=\{\sigma^{-1}(c)\}\times \sigma^{-1}(A)\subseteq T$. We obtain that $(a,a)\in T$, equivalently $(a,a,a)\in S$, contradicting Claim 7.
\end{proof}

\begin{claim}
$\mathrm{pr}_{i,j} S=A^2$ for any $i,j\leq 3$, $i\neq j$.
\end{claim}
\begin{proof}[Proof of the claim]
 We prove it for $(i,j)=(1,3)$, the rest follows by the symmetry of $S$. By the definition of $S$ we have $(a,a),(a,b),(b,a)\in \mathrm{pr}_{1,3}S$, while from Claim 9 follows that $(b,b)\in \mathrm{pr}_{1,3}S$. But these four pairs generate all of $A^2$, so the claim follows from $\mathrm{pr}_{1,3} S\leq \m a^2$.
\end{proof}
Denote by $S^c:=\{(x,y)\in A^2:(c,x,y)\in S\}$. By idempotence, $S^c\leq \m a^2$.
\begin{claim}
For any $c\in A$, $S^c$ is the graph of a bijection.
\end{claim}
\begin{proof}[Proof of the claim]
First of all, Claim 10 and the idempotence of $\m a$ imply that $S^c\leq_{sd}\m a\times\m a$. If $S^c$ is not the graph of an automorphism, by Claim 4 there exists some $d\in A$ such that $d\times A\subseteq S^c$.

Suppose first that $a=c$. In that case, either $d$ is not connected to $a$, or to $b$, in $asm(\m a)$. By the inductive assumption, either $\Sg(a,d)=A$ or $\Sg(b,d)=A$. In the first case, $(a,b),(d,b)\in S^a$ imply $(b,b)\in S^a$, hence $(a,b,b)\in S$. In the second case, $(b,a),(d,a)\in S^a$ imply $(a,a)\in S^a$, hence $(a,a,a)\in S$. Both cases are impossible by Claim 7 and the symmetry of $S$.

Now, let $a\neq c$, and let $S^a$ be the graph of an automorphism of $\m a$. We select some $e\in A$ which is not connected to $d$ in $asm(\m a)$ and such that $e\neq a$. If we were not able to make such a selection, it would mean that $asm(\m a)$ has only two (weakly) connected components, and the one which doesn't contain $d$ is $\{a\}$. Since the automorphism $\varphi$ flips these two components, both of them are of size 1, thus $|A|=2$, which is the base case.

Consider the relation $R_{de}$. By Claim 5, it is the graph of the automorphism $\varphi_{de}$. Using $(c,d,e)\in S$ and (by the symmetry of $S$) $(c,e,d)\in S$, and since $R_{de}=\Sg^{\sm a^2}((d,e),(e,d))$, we obtain $R_{de}\subseteq S^c$. Let $f:=\varphi_{de}(a)$, hence also $\varphi_{de}(f)=a$. We know that $d\neq f$ since $\varphi_{de}(d)=e\neq a$, and from $(f,a)\in R_{de}\subseteq S^c$, we obtain $(f,a)\in S^c$. By the choice of $d$, $(d,a)\in S^c$, so $(c,d,a),(c,f,a)\in S$. By the symmetry of $S$, $(a,c,d),(a,c,f)\in S$, so $S^a$ is not the graph of an automorphism of $\m a$, a contradiction.
\end{proof}

We complete the proof by noticing that, by Claim 11 and the symmetry of $S$, $S$ satisfies the conditions of Theorem~\ref{R3affine}, so $\m a$ is an affine algebra. But then $a\rightarrow_{as}b$ by Base Axiom 1. 
\end{proof}

In the remainder of this section, we turn our attention to directed paths and strong components, so we define some notation and related concepts.

\begin{df}\label{downstreamdef}
For any $x\in\{s,as,sm,asm\}$ we say that there exists a directed $x$-path from $a$ to $b$ if there is a positive integer $k$ and $a=c_0,c_1,\dots,c_k=b$ such that, for all $i<k$, $c_i\rightarrow_x c_{i+1}$. We write $a\rightarrow_x^* b$ when $a=b$ or there exists a directed $x$-path from $a$ to $b$. The relation $\rightarrow_x^*$ is a preorder, and the relation $\{(a,b):a\rightarrow_x^* b$ and $b\rightarrow_x^* a\}$ is an equivalence relation whose blocks are called the strong $x$-components. In view of Proposition~\ref{sameinsubalg}, we need not specify in which algebra we look for the directed path. The strong component of $s(\m a)$ ($as(\m a)$, $sm(\m a)$, $asm(\m a)$) which contains the element $a$ will be denoted $s(a)$ ($as(a)$, $sm(a)$, $asm(a)$).
\end{df}

\begin{df}\label{closeddef}
We say a subset $B\subseteq A\in \vr v$ is {\em s-closed} (\emph{as-closed}, \emph{sm-closed}, \emph{asm-closed}) in $\m a$ if there exists no edge $b\rightarrow_s a$ ($b\rightarrow_{as} a$, $b\rightarrow_{sm} a$, $b\rightarrow_{asm} a$) such that $b\in B$ and $a\in A\setminus B$. A strong s-component (as-component, sm-component, asm-component) which is s-closed (as-closed, sm-closed, asm-closed) is called a sink strong s-component (sink strong as-component, sink strong sm-component, sink strong asm-component). The union of sink strong components of the graphs $s(\m a)$, $as(\m a)$, $sm(\m a)$ and $asm(\m a)$ is denoted by $s-min(\m a)$, $as-min(\m a)$, $sm-min(\m a)$ and $asm-min(\m a)$, respectively.\footnote{In Bulatov's papers, these components and their unions were maximal, but since they will resemble ideals in rings, as we will see, for us it is more natural to call them minimal.}
\end{df}

Our next lemma is a part of Proposition 9.10 from \cite{dreamteam}, though our proof is much more syntactic. In \cite{dreamteam} the authors prove the equivalence, while below we prove only the implication we will use. The opposite direction will follow trivially once we prove another theorem.  

\begin{lm}\label{2abscriterion}
Let $\m a\in \vr v$ and $B\subsetneq A$. Suppose that $B$ is s-closed and also for any $a\in A\setminus B$ and any $b\in B$, there exists a directed s-path from $a$ to some $b'\in B$ which lies entirely within $\Sg(a,b)$. Then $B\lhd_2 A$.
\end{lm}

\begin{proof} Fix the term $f(x,y)$ from Proposition~\ref{universalmeet}. We first prove a claim.
\setcounter{claim}{0}
\begin{claim}
For any $a \in A\setminus B$ and $b \in B$ we can find a term $t_{a,b}(x,y)$ such that for all $c,d\in A$, $c\rightarrow_s^* t_{a,b}(c,d)$ and that $t_{a,b}(a,b), t_{a,b}(b,a)\in B$.
\end{claim}
\begin{proof}[Proof of the claim]
There must be some s-path $a \rightarrow c_1 \rightarrow\dots\rightarrow c_n$ which lies entirely in $\Sg(a,b)$ and such that $c_n \in B$. Let $c_i = q_i(a,b)$ for all $i\leq n$. Now we define inductively the terms $t_i$ as:

\[t_0(x,y) = x\text{ and }t_{i+1}(x,y) := f(t_i(x,y),q_{i+1}(x,y)).\]

Now $t_{a,b}$ is defined to be $t_n$. We note that $t_0(a,b)=a$ and inductively we prove that for all $0<i\leq n$, $t_i(a,b) = c_i$. Indeed, $$t_{i+1}(a,b)=f(t_i(a,b),q_{i+1}(a,b))=f(c_i,c_{i+1})=c_{i+1}.$$ In particular, $t_{a,b} (a,b) = c_n \in B$. On the other hand, note that $t_{a,b}$ is obtained from the variable $x$ by ``multiplying it from the right" with other terms several times, using the binary operation $f$ as the ``multiplication". Applying the second property of $f$ proved in Proposition~\ref{universalmeet} $n$ times, we obtain that $c\rightarrow_s^* t_{a,b}(c,d)$. Substituting $c=b$ and $d=a$ we get that $b\rightarrow_s^* t_{a,b}(b,a)$, and since $B$ is s-closed, it follows that $t_{a,b}(b,a)\in B$, thus proving all statements of the Claim.
\end{proof}

Returning to the main proof, now assume that the term $t(x,y)$ satisfies the following: 
\begin{enumerate}
\item For all $c,d \in A$, $c\rightarrow_s^* t(c,d)$ and 
\item For a maximal number of ordered pairs $(a,b)$ such that $a \in A\setminus B$ and $b \in B$, $t(a,b)\in B$ (we know that $t(b,a) \in B$ for all such pairs by (1) and s-closedness of $B$).
\end{enumerate}

There exists a term that satisfies property (1), such a term is $f(x,y)$, so we can find a term $t(x,y)$ which also satisfies (2).

Assume for some $a,b$ that $a \in A\setminus B$, $b\in B$ and $t(a,b) = a' \in A\setminus B$, while $t(b,a) = b'$. By the property (1), $b'$ must be in $B$. Then set $q(x,y) := t_{a',b'} (t(x,y), t(y,x))$.

First, let $c,d\in A$. Then $c\rightarrow_s^* t(c,d)$. Since $q(c,d)=t_{a',b'} (t(c,d), t(d,c))$, from Claim it follows that $t(c,d)\rightarrow_s^* q(c,d)$. Concatenating the paths, we obtain that $c\rightarrow_s^* q(c,d)$.

Next, suppose that $c\in A\setminus B$ and $d\in B$ are such that $t(c,d)\in B$. Then from Claim applied to $t_{a',b'}$ follows that $t(c,d)\rightarrow_s^* t_{a',b'} (t(c,d), t(d,c))=q(c,d)$, so from s-closedness of $B$ and $t(c,d)\in B$ follows that $q(c,d)\in B$. 

Finally, $$q(a,b) = t_{a',b'} (t(a,b), t(b,a))=t_{a',b'}(a',b')\in B,$$ so $q(x,y)$ contradicts the maximality of $t(x,y)$. This contradiction with the assumption that there exist $a \in A\setminus B$ and $b\in B$ such that $t(a,b) \in A\setminus B$ proves the lemma since, by Theorem \ref{abssetsubuniv}, absorbing sets are also subuniverses.
\end{proof}

Now we need to restate A. Bulatov's Lemma 19 from \cite{BulatovGraph2} and verify that his proof works under our assumptions.

\begin{lm}\label{shifttolchain}
Let $\m a \in \vr v$ and let $S$ be a tolerance\footnote{The lemma holds for any compatible reflexive relation, but we will only use it for tolerances.} of $\m a$. Suppose that $a,b\in A$ are such that $(a, b)$ belongs to the transitive closure of $S$, that is, let $a=c_0,c_1,\dots,c_{k-1},c_k=b$ be such that for all $i<k$, $(c_i,c_{i+1})\in S$. Moreover, suppose that, for some $i\leq k$ and $d_i\in A$, $c_i\rightarrow_s^* d_i$. Then there exist $d_0,d_1,\dots, d_{i-1},d_{i+1},
\dots, d_k\in A$ such that, for all $j\leq k$, $c_j\rightarrow_s^* d_j$ and for all $j<k$, $(d_j,d_{j+1})\in S$.
\end{lm}

\begin{proof}
If $c_i=d_i$ there is nothing to prove, we can just take $d_j:=c_j$ for each $j$. So, suppose that $c_i=e_{i,0}\rightarrow_s e_{i,1}\rightarrow_s\dots\rightarrow_s e_{i,n}=d_i$. For each $j\leq k$ and $l\leq n$, such that $j\neq i$, we define $e_{j,l}$ inductively: $e_{j,0}:=c_j$ and $e_{j,l+1}:=f(e_{j,l},e_{i,l+1})$. Denote $d_j:=e_{j,n}$.

First, it is clear that the definition works also for $j=i$, since $e_{i,l}\rightarrow_s e_{i,l+1}$ and Proposition~\ref{universalmeet} imply $f(e_{i,l},e_{i,l+1})=e_{i,l+1}$.

Next, it is clear by construction that $d_j=f(f(\dots f(c_j,?),\dots,?),?)$, so again using Proposition~\ref{universalmeet} we get that $c_j\rightarrow_s^* d_j$ for each $j\leq k$.

Finally, inducting on $l$, we prove for all $j<k$ and $l\leq n$ that $(e_{j,l},e_{j+1,l})\in S$. The base case $l=0$ was assumed to be true, while from $(e_{j,l},e_{j+1,l})\in S$ and $(e_{i,l+1},e_{i,l+1})\in S$ (which we know by reflexivity of $S$) follows 
\[(e_{j,l+1},e_{j+1,l+1})=(f(e_{j,l},e_{i,l+1}),f(e_{j+1,l},e_{i_1,l+1}))\in S.\]
The above statement, in the case $l=n$, gives $(d_j,d_{j+1})\in S$, finishing the proof.
\end{proof}

The next result we prove is Theorem 21 of \cite{BulatovGraph2} (in \cite{dreamteam}, the related result is Theorem 5.19). We managed to significantly simplify both proofs.

\begin{thm}\label{sminconnected}
Let $\m a\in \vr v$ and $a,b\in s-min(\m a)$. Then $a\rightarrow_{asm}^* b$.
\end{thm}

\begin{proof}
We prove the theorem by an induction on $|A|$. The 2-element case is handled by using Base Axioms 1-3 just like in Theorem~\ref{weakconnect}, since the two semilattices have $|asm-min(\m a)|=1$, while the other two have the graph $asm(\m a)$ strongly connected. So suppose that $|A|>2$ and that the theorem is true for all $\m b\in \vr v$ such that $|B|<|A|$.

Note that $s-min(\m a)\cap asm-min(\m a)\neq\emptyset$. Indeed, take any element $c$ of $asm-min(\m a)$. There exists a directed path $c\rightarrow_s^* d$ to an element $d\in s-min(\m a)$. Since $asm-min(\m a)$ is s-closed, that path never leaves $asm-min(\m a)$, so $d\in s-min(\m a)\cap asm-min(\m a)$.

Instead of proving the statement of the theorem, we will select and fix some $a\in s-min(\m a)\cap asm-min(\m a)$ and prove that every $b\in s-min(\m a)$ satisfies $b\in asm(a)$. If we prove that, the theorem would follow since all of $s-min(\m a)$ would be in the same strong asm-component.
\setcounter{claim}{0}
\begin{claim}
If $c,d\in s-min(\m a)$ and $\Sg(c,d)\neq A$, then $c\rightarrow_{asm}^* d\rightarrow_{asm}^* c$. In particular, if there exist $a'\in s(a)$ and $b'\in s(b)$ such that $\Sg(a',b')\neq A$, then $b\in asm(a)$.
\end{claim}
\begin{proof}[Proof of the claim]
 Let $\m b=\Sg(c,d)$. We can select $c'\in s-min(\m b)$ and $d'\in s-min(\m b)$ such that $c\rightarrow_s^* c'$ and $d\rightarrow_s^* d'$. By the inductive assumption for $\m b$, we know $c'\rightarrow_{asm}^* d'$ and $d'\rightarrow_{asm}^* c'$. Since $c\in s-min(\m a)$ and $c\rightarrow_s^* c'$, it follows that $c'\rightarrow_s^* c$, i.e., $c'\in s(c)$. Analogously, $d'\in s(d)$. Therefore, 
$$
c\rightarrow_s^* c'\rightarrow_{asm}^*d'\rightarrow_s^* d\text{ and }d\rightarrow_s^* d'\rightarrow_{asm}^* c'\rightarrow_s^* c,
$$
proving the first statement of Claim 1. 

In the case when $a'\in s(a)$, $b'\in s(b)$ and $\Sg(a',b')\neq A$, we have that $a',b'\in s-min(\m a)$ and we obtain from the first part of Claim 1 that $a'\rightarrow_{asm}^* b'\rightarrow_{asm}^* a'$. Moreover, we know $a\rightarrow_s^* a'\rightarrow_s^* a$ and $b\rightarrow_s^* b' \rightarrow_s^* b$.  Put the obtained paths together and we obtain
\[a\rightarrow_s^* a'\rightarrow_{asm}^* b'\rightarrow_s^* b\text{ and }b\rightarrow_s^* b'\rightarrow_{asm}^* a'\rightarrow_s^* a,\]
so Claim 1 holds.
\end{proof}
\begin{claim}
If there exists a nontrivial congruence $\alpha\in\Cn a$, then for all $b\in s-min(\m a)$, $b\in asm(a)$.
\end{claim}
\begin{proof}[Proof of the claim]
 By the inductive assumption, $[a]_\alpha\rightarrow_{asm}^* [b]_{\alpha}\rightarrow_{asm}^* [a]_{\alpha}$. Using Homomorphism Axiom 2 several times, we can select some $b'\in [b]_{\alpha}$ so that $a\rightarrow_{asm}^* b'$. $B:=[b]_{\alpha}$ is a proper subuniverse of $\m a$, so $\m b\in \vr v$. Let $b_1,b_1'\in s-min(\m b)$ be such that $b\rightarrow_s^* b_1$ and $b'\rightarrow_s^* b_1'$. By the inductive assumption for $\m b$, $b_1'\rightarrow_{asm}^* b_1$. Since $b\rightarrow_s^* b_1$ and $b\in s-min(\m a)$, thus $b_1\rightarrow_s^* b$ in $\m a$. Concatenating the directed paths we found so far we get
\[a\rightarrow_{asm}^* b'\rightarrow_s^* b_1'\rightarrow_{asm}^* b_1 \rightarrow_s^* b,\]
so $a\rightarrow_{asm}^* b$. The proof of $b\rightarrow_{asm}^* a$ starting from $[b]_{\alpha}\rightarrow_{asm}^* [a]_{\alpha}$ is completely analogous.
\end{proof}

Therefore we suppose that $\m a$ is simple. For each $b\in s-min(\m a)$ we define $R_b:=\Sg^{\sm a^2}((a,b),(b,a))$. Since $\Sg(a,b)=A$, the congruence $lk_1 R_b$ can either be the equality or the full relation.

{\noindent\bf Case 1: $\mathbf{lk_1 R_b=A^2}$ and $\mathbf{R_b[c]=A}$ for some $\mathbf{c\in A}$.}

In this case, we will prove that $b\in asm(a)$. Let $c'\in s-min(\m a)$ be such that $c\rightarrow_s^* c'$. Using the projection homomorphism $\mathrm{pr}_2:\m r_b\rightarrow\m a$, from $(a,c),(b,c)\in R_b$ and Homomorphism Axiom 2 we get that there must exist some $a',b'\in A$ such that in $\m r_b$ we have $(a,c)\rightarrow_s^* (a',c')$ and $(b,c)\rightarrow_s^* (b',c')$. Using Homomorphism Axiom 1 and the homomorphism $\mathrm{pr}_1:\m r_b\rightarrow\m a$ we obtain $a\rightarrow_s^* a'$ and $b\rightarrow_s^* b'$. Since $a,b\in s-min(\m a)$, we get $a'\in s(a)$ and $b'\in s(b)$, so by Claim 1, unless $b\in asm(a)$ immediately, we know that $\Sg(a',b')=A$. Since $R_b[c']\in\Sub(\m a)$ and $(a',c'),(b',c')\in R_b$, it follows that $R_b[c']=A$.

If $B=\Sg(a,c')\neq A\neq \Sg(b,c')=C$, then by Claim 1 we have $a\rightarrow_s^* c'\rightarrow_s^* a$ and $b\rightarrow_s^* c'\rightarrow_s^* b$. Consequently, $b\in asm(a)$.

On the other hand, suppose that $\Sg(a,c')=A$. Then from $(b,a),(b,c')\in R_b$ follows that $R_b[b]=A$, so $(b,b)\in R_b$. As $R_b=\Sg((a,b),(b,a))$, there exists a term $t(x,y)$ such that $t(a,b)=t(b,a)=b$. Together with the idempotence of $t$, this means that $(\{a,b\};t)$ is a semilattice, so by Theorem~\ref{subuniversesminTaylor} and Base Axiom 3, we obtain that $a\rightarrow_s b$. Since $a\in s-min(\m a)$, it follows that $b\rightarrow_s^* a$, so $b\in asm(a)$ in this case. The case when $\Sg(b,c')=A$ is analogous.

{\noindent\bf Case 2: $\mathbf{lk_1 R_b=A^2}$ and $\mathbf{R_b[c]\neq A}$ for all $\mathbf{c\in A}$.} 

In this case we will again prove $b\in asm(a)$. Since $lk_1 R_b=A^2$, the link tolerance $tol_1 R_b$ is connected, and hence there exist $a=c_0,c_1,\dots,c_k=b\in A$ such that, for all $i<k$, $(c_i,c_{i+1})\in tol_1 R_b$. Applying Lemma~\ref{shifttolchain} several times, we can find $d_0,d_1,\dots,d_k\in s-min(\m a)$ such that $c_i\rightarrow_s^* d_i$ for all $0\leq i\leq k$, and also $(d_i,d_{i+1})\in tol_1 R_b$ for all $i<k$. Of course, since $a=c_0\rightarrow_s^* d_0$ and $a\in s-min(\m a)$, thus $d_0\rightarrow_s^* a$, and similarly we also deduce $b=c_k\rightarrow_s^* d_k\rightarrow_s^* b$.

For each $i<k$, $(d_i,d_{i+1})\in tol_1 R_b$, so there exists some $e_i\in A$ such that $d_i,d_{i+1}\in R_b[e_i]$. Given that $R_b[e_i]\neq A$ (or we would be in Case 1), it is a proper subuniverse of $\m a$, and therefore $\Sg(d_i,d_{i+1})\leq R_b[e_i]<\m a$. Since also $d_i,d_{i+1}\in s-min(\m a)$, we apply Claim 1 to them to conclude that $d_i\rightarrow_{asm}^* d_{i+1}\rightarrow_{asm}^* d_i$. Therefore, 
\begin{align*}
a&=c_0\rightarrow_s^* d_0\rightarrow_{asm}^* d_1\rightarrow_{asm}^* \dots\rightarrow_{asm}^* d_k\rightarrow_s^* b\text{ and}\\
b&= c_k\rightarrow_s^* d_k\rightarrow_{asm}^* d_{k-1}\rightarrow_{asm}^* \dots\rightarrow_{asm}^* d_0\rightarrow_s^* a,
\end{align*}
and we again obtain $b\in asm(a)$.

{\noindent\bf Case 3: $\mathbf{lk_1 R_b=0_A}$.}

In this case, we can't immediately deduce $b\in asm(a)$, but we can conclude the following: Either $b\in asm(a)$, or $A=\Sg(a,b)$ and thus $R_b=\Sg((a,b),(b,a))$ is a subdirect subalgebra of $\m a^2$, and therefore $R_b$ is the graph of an automorphism $\varphi_b$ which maps $a$ to $b$ and $b$ to $a$.
\begin{claim}
$s-min(\m a)\subseteq asm-min(\m a)$.
\end{claim}
\begin{proof}[Proof of the claim]
For all $b\in s-min(\m a)$ such that $\Sg(a,b)\neq A$, or $R_b$ is in Cases 1 and 2, we know that $b\in asm(a)\subseteq asm-min(\m a)$. On the other hand, in the case when $R_b$ is the graph of $\varphi_b\in\aut\m a$, Homomorphism Axioms 1 and 2 imply that $\varphi_b$ is also an automorphism of digraphs $as(\m a)$, $sm(\m a)$ and $s(\m a)$. From $a\in s-min(\m a)\cap asm-min(\m a)$ and $\varphi_b(a)=b$ we conclude that also $b\in s-min(\m a)\cap asm-min(\m a)$. Having exhausted all cases of $b\in s-min(\m a)$, we conclude that $s-min(\m a)\subseteq asm-min(\m a)$.
\end{proof}
\begin{claim}
$asm-min(\m a)\lhd_2\m a$.
\end{claim}
\begin{proof}[Proof of the claim]
If $asm-min(\m a)=A$, there is nothing to prove, so suppose that it is not the case. Denote $B:=asm-min(\m a)$. We wish to apply Lemma~\ref{2abscriterion}, so we start by noticing that $B$ is indeed s-closed. 

As for the other condition, suppose that $c,d\in A$ are such that $c\notin B$, while $d\in B$. If $\Sg(c,d)=A$, then there exists some $c'\in s-min(\m a)$ such that $c\rightarrow_s^* c'$. By Claim 3, $c'\in B$ and the path $c\rightarrow_s^* c'$ lies entirely in $A=\Sg(c,d)$, so the condition holds in this case.

If, on the other hand, $\Sg(c,d)=C\neq A$, then there exist some $c',d'\in s-min(\m c)$ such that $c\rightarrow_s^* c'$ and $d\rightarrow_s^* d'$ and the those two s-paths are entirely within $C$. Moreover, since the theorem is true for $\m c$, by the inductive assumption, we know that $d'\rightarrow_{asm}^* c'$. So, we know that $d\rightarrow_s^* d'\rightarrow_{asm}^* c'$, and as $B=asm-min(\m a)$ is asm-closed, it follows that $c'\in B$. So the path $c\rightarrow_s^* c'\in B$ which lies entirely within $C$ proves that $B$ satisfies the conditions of Lemma~\ref{2abscriterion}, so that lemma gives us $B\lhd_2\m a$.
\end{proof}

Because of Claim 4, we know that $asm-min(\m a)$ is a subuniverse of $\m a$. If it is a proper subuniverse, then the theorem is true by Claim 1, as any two elements of $s-min(\m a)\subseteq asm-min(\m a)$ must also generate a proper subuniverse of $\m a$. If, on the other hand, $asm-min(\m a)=A$, then $asm-min(\m a)$ is weakly connected in $asm(\m a)$ by Theorem~\ref{weakconnect}. But any two different sink strong components of $asm(\m a)$ must be disconnected in $asm(\m a)$, so the only possibility is that $asm-min(\m a)$ consists of a single strong component. In that case, any two elements of $s-min(\m a)\subseteq asm-min(\m a)$ are in the same strong component of the digraph $asm(\m a)$ and we are done.

\end{proof}

Along the way to Theorem~\ref{sminconnected}, we proved two facts about $\m a\in \vr v$ which we will list below. As their proof was inside an inductive proof, it is not clear we have them proved thus far, so we reprove them as an easy corollary of Theorem~\ref{sminconnected}.

\begin{cor}\label{singleasmcomp}
Let $\m a\in\vr v$. Then $s-min(\m a)\subseteq asm-min(\m a)$ and moreover, $asm-min(\m a)$ consists of a single strong asm-component.
\end{cor}

\begin{proof}
We proved at the start of the proof of Theorem~\ref{sminconnected} that $s-min(\m a)\cap asm-min(\m a)\neq\emptyset$, while Theorem~\ref{sminconnected} implies that $s-min(\m a)$ is contained within a single strong asm-component, call it $C$. Therefore, $s-min(\m a)\subseteq asm-min(\m a)$. Also, suppose that $a\in asm-min(\m a)$, then there exists a $b\in s-min(\m a)$ such that $a\rightarrow_s^* b$. Since each strong component of $asm-min(\m a)$ is s-closed, thus $b$ is in the same strong asm-component as $a$, and we know that $b\in C$, so $asm-min(\m a)$ consists of just one strong asm-component, namely $C$.
\end{proof}

We can now prove the equivalence of the two conditions in Lemma \ref{2abscriterion}, as we announced. We also provide another characterization of the binary absorption, using the digraph $asm(\m a)$, which is Theorem 5.19 in \cite{dreamteam}\footnote{Note that we use slightly different edges than \cite{dreamteam}: aside from different definitions, their a-edges and m-edges are undirected, only s-edges are directed and coincide with ours.}.

\begin{cor}[Theorem 5.19 and Proposition 9.10 of \cite{dreamteam}]\label{2abscharacterization}
Let $\m a$ be a minimal Taylor algebra and $B\subseteq A$. Then the following are equivalent:
\begin{enumerate}
\item $B\lhd_2 \m a$.
\item $B$ is asm-closed.
\item $B$ is s-closed and for every $a\in A\setminus B$ and every $b\in B$ there exists a directed s-path from $a$ to some $b'\in B$ which lies entirely within $\Sg(a,b)$.
\end{enumerate}
\end{cor}

\begin{proof}
$(1)\Rightarrow(2)$: Assume that $B\lhd_2\m a$. By Theorem~\ref{2absstrproj}, $B$ is a strongly projective subuniverse of $\m a$, so Lemma~\ref{existsaterm} implies that $B$ is asm-closed.

$(2)\Rightarrow (3)$: Suppose that $B$ is asm-closed, hence it is s-closed. For the second condition, assume that $a\in A\setminus B$ and $b\in B$. Since $B$ is asm-closed, hence $B$ is the union of a down-set (order ideal) in the poset of strong asm-components. In any case, $asm-min(\m a)$ is the least strong asm-component, so $asm-min(\m a)\subseteq B$. 

If $\Sg(a,b)=A$, take any $a'\in s-min(\m a)$ such that $a\rightarrow_s^* a'$. We know from Corollary~\ref{singleasmcomp} that $a'\in asm-min(\m a)\subseteq B$, and the whole s-path $a\rightarrow_s^* a'$ lies in $A=\Sg(a,b)$.

On the other hand, if $\Sg(a,b)=C\neq A$, then there exist $a',b'\in s-min(\m c)$ such that $a\rightarrow_s^* a'$, $b\rightarrow_s^* b'$ and both directed paths lie entirely in $C$. Applying Theorem~\ref{sminconnected} to $\m c$, we obtain $b'\rightarrow_{asm}^* a'$. Since $b\in B$, $b\rightarrow_{asm}^* a'$ and $B$ is asm-closed, we obtain that $a'\in B$. But this, together with the path $a\rightarrow_s^* a'$ which lies in $C$ is exactly what is needed to prove $(3)$.

$(3)\Rightarrow (1)$ was established in Lemma~\ref{2abscriterion}.
\end{proof}

\section{M-reduced templates}

In the previous section we described binary absorption. Now we turn to another potential application of binary absorption to solving CSPs. To explain the context, we recall the consistent maps from \cite{SMB2} (first appeared in \cite{M2}, also appearing in \cite{Budich}).

\begin{df}\label{consmaps}
Let $\vr t$ be an M-template and $(V,D,\vr c)$ be a multisorted instance of $CSP(\vr t)$. A set $p = \{\, p_i \mid i\in V \,\}$ of maps is {\em consistent} with $P$ if
for all $i\in V$ the map $p_i$ is a unary polynomial of $\m a_i$, and 
for every constraint $(S,R)$ and tuple $r\in R$ the tuple
$p|_S(r) = \lb p_i(r_i) : i\in S\rb$ is also in $R$.
We say that $p$ is {\em retractive} if $p_i(x)$ is a retraction for all $i\in V$.
Let $p$ be a retractive consistent set of maps. The {\em retraction} of $P$ via $p$ is the new instance $p(P)$ of $CSP(\vr t)$ defined as
\[ p(P) = (\,V, \{p_i(\m a_i) \mid i\in V\,\},\{\ (S,p|_S (R))\mid (S,R)\in \vr c \,\}\,). \]
\end{df}

It easily follows from the definition that for each constraint $(S,R)$, the relation
\[ p|_S(R) = \{\, p|_S(r) \mid r\in R \,\} = R\cap\prod_{i\in S}p_i(A_i) \] 
is a subuniverse of $\prod_{i\in S} p_i(\m a_i)$. Also, if $P$ is $(k,l)$-minimal, then so is its retraction $p(P)$.

\begin{lm}\label{l:retraction}
Let $P$ be a multisorted instance of $CSP(\vr t)$ for some M-template $\vr t$ and let $p$ be a consistent set of retractions. Then $P$ has a solution if and only if $p(P)$ does.
\end{lm}

\begin{df}\label{defeliminated}
We say that an idempotent algebra $\m b$ can be {\em eliminated}, if whenever 
$\vr t$ is an M-template such that $\m B\in\vr t$, and $\vr t\setminus\{\m B\}$ is also a template, and $CSP(\vr t\setminus\{\m B\})$ is tractable, then $CSP(\vr t)$ is also tractable.
\end{df}

One way to prove some algebras can be eliminated is the following lemma (first appeared in the unpublished note \cite{M2}, now also Lemma 4.9 in \cite{SMB2}):

\begin{lm}
\label{l:withC}
Let $\m B$ be an algebra and $t$ be a binary term 
of $\m B$ such that for each $b\in B$ the map $t_b(x) = t(b,x)$ is a retraction which is not surjective.
Let $C$ be the set of elements $c\in B$ such that $x\mapsto t(x,c)$ is a permutation. If $C$ generates a proper subuniverse of $\m B$, then
$\m B$ can be eliminated.
\end{lm}

{\em Sketch of a proof}. Suppose that both $\vr t$ and $\vr t- \{\m b\}$ are templates and $P$ is an instance of $CSP(\vr t)$. The heart of the proof is the construction of a new instance $t(P)$ of $CSP(\vr t - \{\m b\})$ with the following properties:
\begin{itemize}
\item If $P$ has a solution, then so does $t(P)$.
\item If $t(P)$ has a solution, this solution defines a set of retractive consistent maps such that, if $\m a_i=\m b$, then $|p_i(A_i)|<|B|$.
\end{itemize}
Since $\vr t-\{\m b\}$ is closed under retracts, for no $\m a_j$ the algebra $p_j(\m a_j)$ is isomorphic to $\m b$, so $p(P)$ is an instance of $CSP(\vr t - \{\m b\})$. Hence, $t(P)$ can be solved via the algorithm which solves $CSP(\vr t - \{\m b\})$ and, depending on the answer of this algorithm, we either deduce that $P$ has no solution, or find an equivalent instance of $CSP(\vr t - \{\m b\})$ to $P$ which we can solve, again. Thus $\m b$ can be eliminated.
\qed

We say that an M-template $\vr t$ is M-reduced if every maximal-sized algebra in $\vr t$ among those without binary absorbing subuniverses can not be eliminated. ``SMB algebras", which were considered in \cite{SMB2}, are Taylor algebras with two fundamental operations $d(x,y,z)$ and $x\mt y$ which are defined by some equations. In the special case of SMB algebras, the authors were able to apply Lemma~\ref{l:withC} to reduce any instance to an M-reduced instance in which every maximal-sized $\m a_i$ which has a binary absorbing subuniverse has a nice structure. In particular, there was a ``neutral element" $1$ in each such algebra $\m a$ such that, in the language of Bulatov's edges, $1\rightarrow_s x$ for all $x\in A$. Moreover, the element $1$ really acted neutrally: for each term, if we evaluate one of its variables as 1, the obtained polynomial is equal to some other term in one fewer variable.

Using this structure, the authors of \cite{SMB2} proved that an M-reduced template has a solution iff its restriction to variables which have no binary absorbing subuniverses has a solution. In the case of SMB algebras, if $\m a$ has no proper binary absorbing subuniverses, then $\m a$ has no s-edges at all, i.e., there is no subalgebra of $\m a$ which has a proper binary absorbing subuniverse. It has been known since \cite{BKS} that such algebras have ``few subpowers", so they are tractable. This is a shortcut to proving the tractability of SMB algebras, and the nice structure of M-reduced instances was proved using just the binary absorbing part of Zhuk's Reduction Theorem (Theorem~\ref{Zhlemma}). It seemed possible that this binary absorbing part of Zhuk's Reduction Theorem can be proved separately from the rest of the proof of Dichotomy Theorem, and then applied to prove the full Dichotomy just like in the case of SMB algebras.

Thus, we tried to generalize these ideas from SMB algebras to minimal Taylor algebras. Indeed, we were able to prove that any maximal-sized algebra $\m a$ in an M-reduced template has a single maximal (i.e., source) strong component in $asm-(\m a)$ (it also has a single minimal one, by Corollary~\ref{singleasmcomp}). Moreover, we were able to overcome the potential problem with the M-templates, since minimal Taylor algebras might not be closed under taking retracts. By choosing term reducts of those retracts which are minimal Taylor, Definition~\ref{defeliminated} and Lemma~\ref{l:withC} can be modified to apply to minimal Taylor algebras.

However, there exists a minimal Taylor algebra $\m a_1$ which can not be eliminated using Lemma~\ref{l:withC}, but it has more than one element in its sole maximal strong component of $\m a$. The example below appeared first in \cite{Zebnotes}, as Example 4.10.2.

\begin{ex}
The algebra $\m a_1=(\{0,1,2,3\};f)$ is defined below.

\begin{minipage}{0.44\textwidth}
\centering
\includegraphics[width = 1\textwidth]{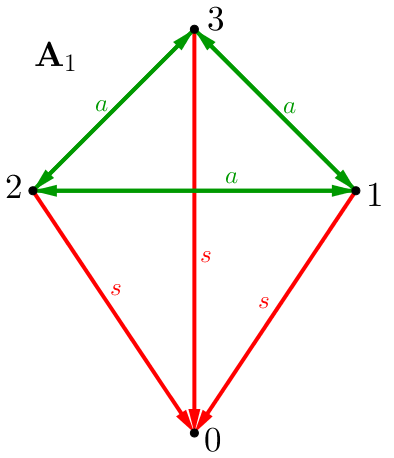}
\end{minipage}
\bigskip
\begin{minipage}{0.5\textwidth}

\begin{flushleft}

$\m A_1$ has the ternary cyclic operation $f$ given by:

\end{flushleft}

\begin{itemize}
\item $f(x,y,z)=0$ when $\{x,y,z\}=\{1,2,3\}$.
\vspace*{5mm}
\item $f(0,x,y)=0$ for all $x,y\in A$.
\vspace*{5mm}
\item On any $\{x,y\}\subseteq \{1,2,3\}$, $f$ restricts as the minority operation.
\end{itemize}

\end{minipage}

It can be proved that $\m a$ is a minimal Taylor algebra, and it has no neutral element. This is a problem since the ideas of \cite{SMB2} involved adding a solution that goes through all neutral elements to the instance, and then applying Zhuk's Reduction Theorem concluding that there must exist some other solution which is never equal to the neutral element. As opposed to a single neutral element which must be a subuniverse (by idempotence), $\{1,2,3\}$ is not a subuniverse, so adding solution(s) which are always in $\{1,2,3\}$ may generate new solutions which are not in $\{1,2,3\}$.
\end{ex}

Now we turn to another way of proving the existence of consistent maps. We need the following auxiliary statement, which is a part of Lemma 49 of \cite{BulatovGraph3}.

\begin{lm}\label{sedgeinjection}
Suppose that $\vr v$ is a pseudovariety of finite minimal Taylor algebras that satisfies Edge Axioms, $\m a\in\vr v$ and $0_{\sm a}\prec\beta$ in $\cn a$ so that $\mathrm{typ}(0_{\sm a},\beta)=\mathbf{2}$. If $C(\eta,\beta;0_{\sm a})$, and $B,C$ are two $\beta$-blocks within the same $\eta$-block such that $B\rightarrow_s C$ in $\m a/\beta$. Then for every $b\in B$ there exists a unique $c\in C$ such that $b\rightarrow_s c$. Moreover, if $b_1\rightarrow c$ and $b_2\rightarrow_s c$ for some $b_1,b_2\in B$ and $c\in C$, then $b_1=b_2$. Finally, if $b\in B$ and $c\in C$ are arbitrary, then $f(b,c)$ is the unique element of $C$ such that $b\rightarrow_s f(b,c)$.
\end{lm}

\begin{proof}
The existence of $c\in C$ such that $b\rightarrow_s c$ follows from Homomorphism Axiom 2. Let $f(x,y)$ be the term which satisfies the conclusions of Proposition~\ref{universalmeet}. 

Suppose that $b\rightarrow_s c_1$ and $b\rightarrow_s c_2$ for some $b\in B$ and $c_1,c_2\in C$. $C$ is a $\beta$-block and from $\mathrm{typ}(0_{\sm a},\beta)=\mathbf{2}$ follows that $C$ is the domain of an affine subuniverse of $\m a$. Stronger Base Axiom 1 implies that $c_1\rightarrow_s c$ holds for no $c\in C$, $c\neq c_1$. On the other hand, we know that $C$ is a subuniverse, so $f(c_1,c_2)\in C$. By Proposition~\ref{universalmeet}, it follows that $f(c_1,c_2)=c_1=f(c_1,c_1)$. By the term condition $C(\eta,\beta;0_{\sm a})$, it follows that $f(b,c_1)=f(b,c_2)$. But the assumptions $b\rightarrow_s c_1$ and $b\rightarrow_s c_2$, together with Proposition~\ref{universalmeet} imply that $f(b,c_1)=c_1$ and $f(b,c_2)=c_2$, so we get $c_1=c_2$.

Next, let $b_1\rightarrow_s c$ and $b_2\rightarrow_s c$ for some $b_1,b_2\in B$ and $c\in C$. From Proposition~\ref{universalmeet} we get that $f(b_1,c)=c=f(b_2,c)$. The term condition $C(\eta,\beta;0_{\sm a})$ implies that $f(b_1,b_2)=f(b_2,b_2)=b_2$. Analogously as above, since $B$ is the domain of an affine subalgebra of $\m a$, we conclude $f(b_1,b_2)=b_1$, so $b_1=b_2$.

Finally, under the assumptions of the last sentence, $f(b,c)\in C$, so $f(b,c)\neq b$. Therefore, by Proposition~\ref{universalmeet}, $b\rightarrow_s f(b,c)$ and by the previous parts the element $f(b,c)$ is uniquely determined in $C$.
\end{proof}

The next result, taken from \cite{Budich} together with its proof, considers multisorted CSP instances in which each $\m a_i$ is a subdirectly irreducible algebra. Any instance can be equivalently transformed into such an instance by using Birkhoff's Subdirect Decomposition Theorem, while decreasing the domains of variables which are changed, so this assumption comes at no cost. We will call an instance in which all domains of variables are subdirectly irreducible an {\em SI instance}. In a 1-minimal SI instance, we say that the domain of variable $\m a_i$ is a {\em large centralizer} domain if the centralizer condition $C(1_{\sm a_i},\mu_i;0_{\sm a_i})$ holds.

\begin{thm}\label{largecentred}
Let $\vr t$ be an M-template of minimal Taylor algebras and $P$ an SI-instance of $CSP(\vr t)$. Let $P/\overline{\mu}$ be the instance obtained from $P$ by factoring every large centralizer domain of variable $\m a_i$ by its monolith congruence $\mu_i$, and leaving all other domains of variables unchanged. Suppose that for any point $a$ in any domain of $P/\overline{\mu}$ (this point $a$ may be either an element of $A_i$, or a $\mu_i$-class, depending on whether $\m a_i$ is large centralizer domain), $P/\overline{\mu}$ has a solution $f_a$ such that $f(i)=a$. Then there exists a retractive set of consistent maps such that each large centralizer domain which has an s-edge is mapped onto a proper retract.

In particular, under above assumptions, if $\m a_i$ is a large centralizer domain and $\m b_i\lhd_2\m a_i$, then $P$ has a solution iff $P$ has a solution such that $f(i)\in B_i$.
\end{thm}

\begin{proof}
Let all large centralizer domains which have s-edges be $\{\m a_i:1\leq i\leq k\}$ and for each $i\leq k$ select $a_i\rightarrow_s b_i$, $a_i\neq b_i$. Let $g_i$ be a solution of $P/\overline{\mu}$ such that $g_i(i)=[b_i]_{\mu_i}$. We define mappings $h_i\in \prod_{i\in V}A_i$ by: For each variable $j\in V$ such that $\m a_j$ is a large centralizer domain, select $h_i(j)\in g_i(j)$, while for each $j\in V$ such that $\m a_j$ is not a large centralizer domain (and hence $g_i(j)$ is an element of $A_j$), let $h_i(j)=g_i(j)$. Also, denote by $f(x,y)$ the binary term from Proposition \ref{universalmeet}. For any variable $i$, let 
\[p_j'(x)=
f(\dots f(f(x,h_1(j)),h_2(j)),\dots,h_k(j))
\]
and let $p=\{p_j:j\in V\}$ be the $M!$-composition of $p'$ with itself, where $M$ is a positive integer such that, for all $j\in V$, $M\geq |A_j|$. In other words, for all $j\in V$,
\[
p_j(x)=(\underbrace{p_j'\circ p_j'\circ\dots\circ p_j'}\limits_{M!-1\text{ signs }\circ})(x)
\]
By construction, it is clear that, for each $j\in V$, $p_j(x)\in\poln{1}{a_j}$ and that $p_j(p_j(x))=p_j(x)$ for all $x\in A_j$. 

To prove that $\{p_j:j\in V\}$ is a set of consistent maps, it suffices to prove the same for $\{p_j':j\in V\}$. Even more incrementally, to prove that $\{p_j:j\in V\}$ is a set of consistent maps, we need to prove the following 
\setcounter{claim}{0}
\begin{claim}
For every constraint $(S,R)$, every $i\in V$ and every $\mathbf{a}\in R$, there exists some $\mathbf{b}\in R$ such that, for every $j\in S$, $f(\mathbf{a}(j),\mathbf{b}(j))=f(\mathbf{a}(j),h_i(j))$.
\end{claim}
\begin{proof}[Proof of the claim]
 By the fact that $g_i$ is a solution of $P/\overline{\mu}$ and the choice of $h_i$, we can select $\mathbf{b}\in R$ such that
\begin{itemize}
	\item For every $j\in S$ such that $\m a_j$ is a large centralizer domain, $g_i(j)$ is a $\mu_j$-class, while $\mathbf{b}(j)$ and $h_i(j)$ are both in $g_i(j)$, and
	\item For every $j\in S$ such that $\m a_j$ is not a large centralizer domain, $\mathbf{b}(j)=h_i(j)= g_i(j)$.
\end{itemize}

For every $j\in S$ such that $\m a_j$ is not a large centralizer domain, we immediately obtain that $f(\mathbf{a}(j),\mathbf{b}(j))=f(\mathbf{a}(j),h_i(j))$.

On the other hand, for every $j\in S$ such that $\m a_j$ is a large centralizer domain, we consider s-edges in $\m a_j$. From the monotonicity of the centralizer and $C(1_{\sm a_j},\mu_j;0_{\sm a_j})$ follows that $C(\mu_j,\mu_j;0_{\sm a_j})$. Hence $\mu_j$ is an Abelian congruence, and since $\m a_j$ is both Taylor and idempotent, each $\mu_j$-class is an affine algebra. According to Stronger Base Axiom 1, there are no s-edges between different elements of any $\mu_j$-class. So any s-edges in $\m a_j$ are between elements of two different $\mu_j$-blocks, and by Homomorphism Axiom 1 there is an s-edge in $\m a_j/\mu_j$ between those two blocks.

By Proposition~\ref{universalmeet}, there are only two possibilities for $f(\mathbf{a}(j),\mathbf{b}(j))$: 
\[
f(\mathbf{a}(j),\mathbf{b}(j))=\mathbf{a}(j), \text{ or } \mathbf{a}(j)\rightarrow_s f(\mathbf{a}(j),\mathbf{b}(j)). 
\]

If $f(\mathbf{a}(j),\mathbf{b}(j))=\mathbf{a}(j)$, then 
\[[f(\mathbf{a}(j),h_i(j))]_{\mu_j}=[f(\mathbf{a}(j),\mathbf{b}(j))]_{\mu_j}=[\mathbf{a}(j)]_{\mu_j}.\]
We proved above that $\mathbf{a}(j)\rightarrow_s f(\mathbf{a}(j),h_i(j))$ is impossible since they are in the same $\mu_j$-class, so the remaining option is $f(\mathbf{a}(j),h_i(j))=\mathbf{a}(j)=f(\mathbf{a}(j),\mathbf{b}(j))$.

If, on the other hand, $\mathbf{a}(j)\rightarrow_s f(\mathbf{a}(j),\mathbf{b}(j))$, then $B=[\mathbf{a}(j)]_{\mu_j}$ and $C=[f(\mathbf{a}(j),\mathbf{b}(j))]_{\mu_j}$ are distinct $\mu_j$-blocks, and $B\rightarrow_s C$ in $\m a_j/\mu_j$. Moreover, from $[\mathbf{b}(j)]_{\mu_j}=[h_i(j)]_{\mu_j}$ follows that $[f(\mathbf{a}(j),h_i(j))]_{\mu_j}=C$. Hence, $\mathbf{a}(j)\rightarrow_s f(\mathbf{a}(j),h_i(j))$, too. According to Lemma~\ref{sedgeinjection}, there is a unique element $c\in C$ such that $\mathbf{a}(j)\rightarrow_s c$, and therefore $f(\mathbf{a}(j),\mathbf{b}(j))=c=f(\mathbf{a}(j),h_i(j))$.
\end{proof}

By the Claim above and using $f(\mathbf{a},\mathbf{b})\in R$, we have that $p$ is a retractive set of consistent maps.

Next we prove, for all $j\in V$ such that $\m a_j$ is a large centralizer domain, that $p_j(A_j)\subsetneq A_j$. Since $p_j$ is a composition of a sequence of unary polynomials of the form $f(x,h_i(j))$ acting on a finite set, it suffices to prove that at least one of the polynomials in this sequence is not bijective. But for the polynomial $f(x,h_j(j))$ we know that $[h_j(j)]_{\mu_j}=[b_j]_{\mu_j}$ and Homomorphism Axiom 1 implies that $[a_j]_{\mu_j}\rightarrow_s [b_j]_{\mu_j}=[h_j(j)]_{\mu_j}$. Lemmas~\ref{universalmeet} and \ref{sedgeinjection} imply that $b_j=f(a_j,b_j)=f(a_j,h_j(j))$, while the argument we made above proves that there can be no s-edges within the same $\mu_j$-class in large centralizer domains, and hence $f(b_j,h_j(j))=b_j$. So $f(a_j,h_j(j))=f(b_j,h_j(j))$ and $f(x,h_j(j))$ is not a bijection of $A_j$.

For the final sentence we may assume $B_i\subsetneq A_i$. Next, note that for all $a\in A_i\setminus B_i$ and $b\in B_i$, $a\rightarrow_s f(a,b)\in B_i$. So we may select $h_i(i)=b_i$ and the polynomial $f(x,h_i(i))$ will map $A_i$ into $B_i$. The polynomial $p_i$ is constructed by composing successively polynomials of the form $f(x,h_j(i))$, and from $\m b_i\lhd_2 \m a_i$ we get that, after the first application of the polynomial $f(x,h_i(i))$, the image of $f(x,h_j(i))$ will never leave $B_i$. Therefore, $p_i$ is a retraction of $\m a_i$ whose image is a subset of $B_i$. If $P$ has a solution $g$, by applying the consistent set of maps $\{p_j:j\in V\}$ coordinatewise to $g$, we obtain a new solution $f$ which is guaranteed to satisfy $f(i)=p_i(g(i))\in B_i$.
\end{proof}

We repeated the proof of Theorem \ref{largecentred} in order to show that the last sentence follows trivially from its proof. This sentence is a special case of Zhuk's Reduction Theorem, so we established that Bulatov's result can be applied to prove Zhuk's Reduction Theorem in this special case.

\section{Concluding remarks and further research}

In this paper we introduced a different take on Bulatov's colored edges, by focusing on the properties such objects ought to have and then finding various graphs which fit these properties. We managed to reprove some of the results in the initial part of Bulatov's theory just from these abstract properties, and there is no doubt in our minds that the whole theory can be restated in this way. Next, we simplified the proofs of Theorem 33, Theorem 36 and Corollary 38 from their original proofs in \cite{BulatovGraph1}, \cite{BulatovGraph2} and \cite{dreamteam}. Finally, we analyzed the approach via consistent maps by M. Mar\'oti and found reasons why it does not generalize fully from SMB algebras to minimal Taylor algebras. However, we managed to apply Bulatov's result in the case of ``large centralizers", with subdirectly irreducible domains of variables, to prove the binary absorbing part of Zhuk's Reduction Theorem in that case.

In subsequent projects, some authors of this paper together with G. Gyenizse and M. Kozik proceeded to investigate the two proofs of the Dichotomy Theorem, and discovered further connections between them. However, since colored edges are shown to be less useful in these subsequent results, much less than in investigating binary absorption, we split these results off into a separate paper, which is in production. More recently, another subset of authors of this paper, together with M. Mar\'oti, managed to prove the Dichotomy Theorem using consistent maps and a stronger version of Zhuk's Reduction Theorem, but only its part related to binary absorption.

\section*{Acknowledgement}

We want to express our gratitude to Professor Marcin Kozik. Our discussion with him in the Spring of 2022 set us on the path towards this paper, and all subsequent results which will be in the two papers announced above. His actual proofs will all be published in the next paper in this series, but none of the three papers would exist without his initial input.

\section*{Declarations}

\subsection*{Author contributions}

This is original research, and all authors have contributed to it. The results included in this paper have not been submitted to, or published in, other journals.

\subsection*{Ethical statement}

Ethical approval is not needed for this research since it is purely theoretical.

\subsection*{Conflict of interest statement}

The authors declare that they have no conflicts of interest. 

\subsection*{Data availability statement}

Data sharing is not applicable to this article as datasets were neither generated nor analyzed.

\subsection*{Funding statement}

We have listed all sources of financial support in the first page.

\end{document}